\documentclass[a4paper,reqno,11pt]{amsart}
\usepackage[T1]{fontenc}
\usepackage{amsmath,amsfonts,amssymb,extarrows,thmtools}
\usepackage{stmaryrd}
\usepackage{dsfont}
\usepackage[shortlabels]{enumitem}
\usepackage{mathtools}
\usepackage{mathrsfs}
\usepackage{comment}
\usepackage[margin=1in]{geometry}
\usepackage{xcolor}
\definecolor{hypercolor}{rgb}{0.5,0,0.5}
\usepackage[colorlinks,linkcolor=hypercolor,citecolor=hypercolor,pagebackref]{hyperref}
\usepackage{tikz-cd}

\theoremstyle{plain}
\newtheorem{theorem}{Theorem}[subsection]
\newtheorem{proposition}[theorem]{Proposition}
\newtheorem{lemma}[theorem]{Lemma}
\newtheorem{corollary}[theorem]{Corollary}
\newtheorem{conjecture}[theorem]{Conjecture}
\theoremstyle{definition}
\newtheorem{definition}[theorem]{Definition}
\newtheorem{example}[theorem]{Example}

\newtheorem{remark}[theorem]{Remark}

\numberwithin{equation}{subsection}
\DeclareMathOperator{\crys}{crys}
\DeclareMathOperator{\dR}{dR}
\DeclareMathOperator{\CH}{CH}
\DeclareMathOperator{\Gal}{Gal}
\DeclareMathOperator{\Br}{Br}
\DeclareMathOperator{\fl}{fl}
\DeclareMathOperator{\D}{D}
\DeclareMathOperator{\id}{id}
\DeclareMathOperator{\Hom}{Hom}
\DeclareMathOperator{\Ext}{Ext}
\DeclareMathOperator{\Km}{Km}
\DeclareMathOperator{\SO}{SO}
\DeclareMathOperator{\SL}{SL}
\DeclareMathOperator{\Spin}{Spin}

\DeclareMathOperator{\spec}{Spec}

\DeclareMathOperator{\alg}{alg}
\DeclareMathOperator{\num}{num}

\DeclareMathOperator{\ch}{ch}
\DeclareMathOperator{\td}{Td}
\DeclareMathOperator{\Iso}{Isom}

\DeclareMathOperator{\cl}{cl}
\DeclareMathOperator{\R}{R}

\DeclareMathOperator{\cha}{char}

\DeclareMathOperator{\SN}{SN}
\DeclareMathOperator{\ad}{ad}
\DeclareMathOperator{\even}{even}

\DeclareMathOperator{\Og}{O}

\DeclareMathOperator{\tr}{tr}
\DeclareMathOperator{\fil}{Fil}

\DeclareMathOperator{\disc}{disc}

\DeclareMathOperator{\hdg}{Hdg}
\DeclareMathOperator{\Gm}{\mathbb{G}_m}

\DeclareMathOperator{\Sch}{\mathtt{Sch}}

\newcommand{\sX}{\mathscr{X}}
\newcommand{\sY}{\mathscr{Y}}
\newcommand{\srM}{\mathscr{M}}
\newcommand{\sM}{\mathscr{M}}
\newcommand{\Zl}{\mathbb{Z}_{\ell}}
\newcommand{\Zp}{\mathbb{Z}_p}
\newcommand{\Ql}{\mathbb{Q}_{\ell}}

\newcommand{\bbZ}{\mathbb{Z}}
\newcommand{\bbC}{\mathbb{C}}

\newcommand\cE{\mathcal{E}}
\newcommand\cF{\mathcal{F}}

\newcommand\cK{\mathcal{K}}
\newcommand\cL{\mathcal{L}}

\newcommand\cO{\mathcal{O}}
\newcommand\cP{\mathcal{P}}

\newcommand\cX{\mathcal{X}}
\newcommand\cY{\mathcal{Y}}

\renewcommand\AA{\mathbb{A}}

\newcommand\CC{\mathbb{C}}

\newcommand\FF{\mathbb{F}}
\newcommand\GG{\mathbb{G}}

\newcommand\QQ{\mathbb{Q}}

\newcommand\ZZ{\mathbb{Z}}

\newcommand\bB{\mathbf{B}}

\newcommand\bR{\mathbf{R}}

\newcommand\rD{\mathrm{D}}

\newcommand\rH{\mathrm{H}}

\newcommand\rK{\mathrm{K}}

\newcommand\rN{\mathrm{N}}
\newcommand\rO{\mathrm{O}}

\newcommand\rT{\mathrm{T}}

\newcommand\rV{\mathrm{V}}

\newcommand\fX{\mathfrak{X}}
\newcommand\fY{\mathfrak{Y}}

\newcommand\frh{\mathfrak{h}}

\newcommand\ord{{\rm ord}}
\newcommand\rank{{\rm rank}}

\newcommand\Pic{{\rm Pic}}

\newcommand\NS{{\rm NS}}

\newcommand\Gr{{\rm Gr}}
\newcommand\GL{{\rm GL}}

\newcommand{\Spec}{ \mathrm{Spec}}

\newcommand{\abelian}{\mathtt{AV}}
\newcommand{\coh}{\mathtt{Coh}}
\DeclareMathOperator{\et}{{\acute{e}t}}

\DeclarePairedDelimiterX\set[1]\lbrace\rbrace{\def\given{\;\delimsize\vert\;}#1}
\DeclarePairedDelimiter{\pair}{\langle}{\rangle}

\newcommand{\cf}{\textit{cf.~}}

\title{Derived isogenies and isogenies for abelian surfaces}
\author{Zhiyuan Li}
\address{SCMS, Fudan University, No,2005 Songhu Road, Shanghai, China}
\email{zhiyuan\_li@fudan.edu.cn}
\author{Haitao Zou}
\address{Universit\"at Bielefeld,
Universit\"atsstraße 25, 33615, Bielefeld, Germany}
\email{hzou@math.uni-bielefeld.de}

\subjclass[2020]{Primary 14F08, 14K02; Secondary 14G17}
\keywords{abelian surface, isogenies, derived categories, twisted sheaves, Torelli theorems}

\thanks{This project is supported by the NKRD Program of China (No.~2020YFA0713200),  NSFC (No.~12121001, No.~12171090 and No.~12425105)  and Shanghai Pilot Program for Basic Research (No.~21TQ00). Z. Li is also a member of LMNS. H. Zou is also supported by the Deutsche Forschungsgemeinschaft (DFG, German Research Foundation) -- Project-ID 491392403 -- TRR 358}
\begin{document}
\begin{abstract}
In this paper, we study the twisted Fourier--Mukai partners of abelian surfaces. Following the work of Huybrechts \cite{Hu19}, we introduce the twisted derived equivalences (also called derived isogenies) between abelian surfaces. We show that there is a twisted derived Torelli theorem for abelian surfaces over algebraically closed fields with characteristic $\neq 2,3$.
For this we firstly extend a trick given by Shioda on integral Hodge structures, to rational Hodge structures, $\ell$-adic Tate modules and $F$-crystals. Using this trick, we can confirm the Tate conjecture in a special case.
Then we make use of Tate's isogeny theorem to  give a characterization of the derived isogenies between abelian surfaces via  so called principal isogenies. As a consequence, we show the two abelian surfaces are principally isogenous if and only if they are derived isogenous.
\end{abstract}
\maketitle

\tableofcontents
\section{Introduction}
\subsection{Background}In the study of abelian varieties, a natural question is to classify the Fourier--Mukai partners of abelian varieties.  Due to Orlov and Polishchuk's \emph{derived Torelli theorem} for abelian varieties in (\cf\cite{orlov2002,polishchuk96}), there is a geometric/cohomological  classification of derived equivalences between them.
More generally, one can consider the {\it twisted derived equivalence} or so called {\it derived isogeny} between abelian varieties  in the spirit of \cite{Hu19}. 
\begin{definition}
    Two abelian varieties $X$ and $Y$  are derived isogenous if they can be connected by derived equivalences between twisted abelian varieties, i.e. there exist twisted abelian varieties $(X_i,\alpha_i)$ and $(X_i,\beta_i)$ such that there is a sequence of derived equivalences
\begin{equation}\label{eq:zigzagtwsted}
\begin{tikzcd}[row sep=0cm, column sep=small]
\D^b(X,\alpha)\ar[r,"\simeq"] &\D^b(X_1,\beta_1) &  &\\
& \D^b(X_1,\alpha_2) \ar[r,"\simeq"]&  \D^b(X_2,\beta_2) &\\
& & \vdots & \\
& & \D^b(X_n,\alpha_{n+1}) \ar[r,"\simeq"] & \D^b(Y,\beta_n)
\end{tikzcd}
\end{equation}
where $\D^b(X,\alpha)$ is the bounded derived category of $\alpha$-twisted coherent sheaves on $X$. 
\end{definition}

In \cite{Stellari2007}, Stellari proved that derived isogenous complex abelian surfaces are isogenous using the the Kuga--Satake varieties associated to their transcendental lattices (\cf Theorem 1.2 in loc. cit. ). However, the converse is not true as there are isogenous abelian surfaces which are not derived isogenous (\cf Remark 4.4 (ii) in loc. cit. ).  The main goal of this paper is to give a cohomological and geometric classification of  derived isogenies between abelian surfaces over algebraically closed fields of arbitrary characteristic.  

\subsection{Twisted derived Torelli theorem  for abelian surfaces in characteristic zero} 
Let us first classify the derived isogenies between abelian surfaces in term of isogenies.  For this purpose, we need to introduce a new type of isogeny: We say two abelian surfaces $X$ and $Y$ are  \emph{principally isogenous} if there is an isogeny $f$ from $X$ to $Y$ of square degree. For example, $X$ and its dual abelian variety $\widehat{X}$ are principally isogenous since any polarization $\cL$ on $X$ induces an isogeny $f_{\cL} \colon X \to \widehat{X}$ of degree $\chi(\cL)^2$.

The first main result is %

\begin{theorem}\label{TDGC} Let $X$ and $Y$ be two abelian surfaces over $k=\bar{k}$ with $\cha(k)=0$. The following statements are equivalent.
\begin{enumerate}[label={\rm(\roman*)},leftmargin=\parindent,labelindent=\parindent,align=left]
\item $X$ and $Y$ are derived isogenous.
\item $X$ and $Y$ are principally isogenous.
\end{enumerate}
 
\end{theorem}
A notable fact for abelian surfaces is that besides their $1^{st}$ cohomology groups, their $2^{nd}$ cohomology groups also carry rich structures.  In the untwisted case, Mukai and Orlov have showed \cite{mukai84, orlov2002} that  \begin{center}
   $\rD^b(X)\cong \rD^b(Y) \Leftrightarrow \widetilde{\rH}(X,\ZZ)\cong_{\hdg}  \widetilde{\rH}(Y,\ZZ)\Leftrightarrow \rT(X)\cong_{\hdg} \rT(Y),$ 
\end{center}  
where $\widetilde{\rH}(X,\ZZ)$ and $\widetilde{\rH}(Y,\ZZ)$ are the Mukai lattices, $\rT(X)\subseteq \rH^2(X,\ZZ)$ and $\rT(Y)\subseteq \rH^2(Y,\ZZ)$ denote the transcendental lattices, $\cong_{\hdg}$ means integral Hodge isometries (\cf\cite[Theorem 5.1]{BridgelandMaciocia}).  The following result can be viewed as a generalization of Mukai and Orlov's  result. 

\begin{corollary}\label{cor:TDGC}
The statement (i) and (ii) of Theorem \ref{TDGC} is also equivalent to the following equivalent conditions
\begin{enumerate}[start=3,label={\rm(\roman*)},leftmargin=\parindent,labelindent=\parindent,align=left]
\item the associated Kummer surfaces $\Km(X)$ and $\Km(Y)$ are derived isogenous;
\item Chow motives $\frh(X)\cong \frh(Y)$ are isomorphic as Frobenius exterior algebras;
\item  even degree Chow motives $\frh^{\even}(X)\cong \frh^{\even}(Y)$ are isomorphic as Frobenius algebra.
\end{enumerate}
When $k=\CC$,  then the conditions above are also equivalent to 
\begin{enumerate}[resume*]
    \item $\rH^2(X,\QQ)\cong \rH^2(Y,\QQ)$ as a rational Hodge isometry;
    \item  $\widetilde{\rH}(X,\QQ)\cong \widetilde{\rH}(Y,\QQ) $ as a rational Hodge isometry;
     \item $\rT(X)\otimes \QQ\cong \rT(Y)\otimes \QQ$ as a rational Hodge isometry. 
\end{enumerate}
\end{corollary}

Here, the motive $\frh(X)$ admits a canonical motivic decomposition produced by Deninger--Murre \cite{DeningerMurre91}
\begin{equation}\label{eq:motivicdecomposition}
\frh(X)=\bigoplus_{i=0}^{4}\frh^{i}(X)
\end{equation}
such that $\rH^*(\frh^i(X)) \cong \rH^i(X)$ for any Weil cohomology $\rH^*(-)$. It satisfies $\frh^{i}(X)=\bigwedge^{i}\frh^{1}(X)$ for all $i$, $\frh^{4}(X)\simeq \mathds{1}(-4)$ and $\bigwedge^{i}\frh^{1}(X)=0$ for $i>4$ (\cf\cite{MR1265530}). The motive $\frh(X)$ is a Frobenius exterior algebra objects in the category of Chow motives over $k$ and the even degree part 
\begin{equation}\label{eq:motivcevendecomposition}
\frh^{\even}(X)= \bigoplus_{k =0}^2 \bigwedge^{2k} \frh^1(X)
\end{equation}
forms a Frobenius algebra object in the sense of \cite{LV19}. %

The equivalences $(i) \Leftrightarrow (iv) \Leftrightarrow (v)$ are motivic realizations of derived isogenies between abelian surfaces, which can be viewed as an analogy of the motivic global Torelli theorem on K3 surfaces (\cf\cite[Conjecture 0.3]{Hu19} and \cite[Theorem 1]{LV19}).  The equivalences  $(i) \Leftrightarrow (iii) \Leftrightarrow (viii)$ can be viewed as  a generalization of  \cite[Theorem 1.2]{Stellari2007}. The Hodge-theoretic realization $(i) \Leftrightarrow (vi)$   follows a similar strategy  of \cite[Theorem 0.1]{Hu19}, which makes use of Shioda's period map and Cartan--Dieudonn\'e decomposition of a rational isometry.  The  equivalences $(vi) \Leftrightarrow (vii) \Leftrightarrow (viii)$ follow from the Witt cancellation theorem (see \S\ref{subsec:proofTDGC}).

\subsection{Shioda's trick} The proof of Theorem \ref{TDGC} is concluded by a new ingredient so called rational Shioda's trick on abelian surfaces. The original Shioda's trick in \cite{Shioda78}  plays a key role in the proof of Shioda's global Torelli theorem for abelian surfaces, which links the weight-1 integral  Hodge structure to the weight-2 integral Hodge structure of an abelian surface. We generalize it in the following form.
 \begin{theorem}[Shioda's trick, see \S\ref{section:shiodatrick}]\label{Hodge-derivediso}
 Let $X$ and $Y$ be two complex abelian surfaces. Then for any admissible Hodge isometry $$\psi\colon \rH^2(X,\QQ)\xrightarrow{\sim} \rH^2(Y,\QQ)$$ we can find an isogeny $f \colon Y\to X$ of degree $d^2$ such that $\psi = \frac{f^*}{d}$.
 \end{theorem}
As an application, the generalized Shioda's trick gives the algebraicity of some cohomological cycles. For any integer $d$, one can consider a Hodge similitude of degree $d$
\[ \rH^2(X,\QQ)\xrightarrow{\sim} \rH^2(Y,\QQ),\]
called a \textit{Hodge isogeny of degree $d$}. Due to the Hodge conjecture on products of abelian surfaces,  we know that every Hodge isogeny is algebraic.  Our generalized Shioda's trick actually shows that it is induced by certain isogenies.  Similarly, we prove the  $\ell$-adic and $p$-adic Shioda's trick, which gives a proof of Tate conjecture for isometries between the $2^{nd}$-cohomology groups (as either Galois-modules or crystals) of abelian surfaces over finitely generated fields. See Corollary \ref{cor:shiodatrickforisog} for more details.

\subsection{Results in positive characteristic}The second part of this paper is to investigate the twisted derived Torelli theorem over  positive characteristic fields.
Due to the absence of a satisfactory global Torelli theorem, one cannot follow the argument in characteristic zero directly.  Instead, we need some input from $p$-adic Hodge theory. Our formulation is the following.
\begin{theorem}\label{derivediso}
Let $X$ and $Y$ be two abelian surfaces over  $k=\bar{k}$ with $\cha(k)=p>3$. Then the following statements are equivalent.
\begin{enumerate}[label={\rm(\roman*$'$)},leftmargin=\parindent,labelindent=\parindent,align=left]
\item  $X$ and $Y$  are  prime-to-$p$ derived isogenous.
\item  $X$ and $Y$ are prime-to-$p$ principally isogenous. 
\end{enumerate}
Moreover, in case that $X$ is supersingular, then $Y$ is derived isogenous to $X$ if and only if $Y$ is supersingular.
 \end{theorem}
 
Here, we say a derived isogeny as \eqref{eq:zigzagtwsted} is {\it prime-to-$p$} if its crystalline realization is integral (see Definition \ref{def:primetopderived} for details), which is a condition somewhat technical.  The main ingredients in the proof of Theorem \ref{derivediso} are the lifting-specialization technique, which works well for prime-to-$p$ derived isogenies. Actually, our method shows that there is an implication  $(i')\Rightarrow (ii')$ for derived isogenies which are not necessarily being prime-to-$p$ (see Theorem \ref{genderivediso}). Conversely,  we believe that the existence of  quasi-liftable  isogenies will imply the existence of derived isogeny (see Conjecture \ref{gen-derivedconj}). The only obstruction is the existence of the specialization of non-prime-to-$p$ derived isogenies between abelian surfaces. See Remark \ref{rmk:specialization}.

Another natural question is  whether two abelian surfaces are derived isogenous if and only if their associated Kummer surfaces are derived isogenous over positive characteristic fields. Unfortunately,  we cannot fully prove the equivalence. Instead, we provide a  partial solution of this question.  See Theorem \ref{kumm-isogeny} for more details. 

 Similarly, one may ask whether such results also hold for K3 surfaces.  Let $\FF_q$ be a finite field with $q = p^k$ a power of some prime $p$. Recall that two K3 surfaces $S$ and $S'$ over $\FF_q$ are (geometrically) isogenous in the sense of \cite{Yang2020} if there exists an algebraic correspondence $\Gamma$ which induces an isometry of $\Gal(\bar{\FF}_p/k)$-modules
\[
\Gamma^\ast_\ell\colon \rH^2_{\et}(S_{\bar{\FF}_p},\QQ_\ell)\xrightarrow{\sim}\rH^2_{\et}(S'_{\bar{\FF}_p},\QQ_\ell),
\]
for all $\ell\nmid p$ and  an isometry of isocrystals
\[
\Gamma^\ast_p\colon \rH^2_{\crys}(S_{k}/K)\xrightarrow{\sim}\rH^2_{\crys}(S'_{k}/K),
\]
for some finite extension $k/\FF_q$ and the fraction field $K$ of $W=W(k)$. More generally, we can take a finitely generated field $k$ over $\FF_p$ and a Cohen ring of $k$. Then we say that the isogeny is prime-to-$p$ if the isometry $\Gamma^\ast_p$ is integral, i.e., $\Gamma^{\ast}_p\left(\rH^2_{\crys}(S_{k}/W) \right)=\rH^2_{\crys}(S'_{k}/W)$. This leads us to a formulation of the twisted derived Torelli conjecture for K3 surfaces. 
\begin{conjecture}\label{mainconj}
For two K3 surfaces $S$ and $S'$ over a finitely generated field $k$, then the following are equivalent.
\begin{enumerate}[label={\rm (\alph*)}]
    \item There exists a derived isogeny $\D^b(S) \sim \D^b(S')$.
    \item There exists an isogeny between $S$ and $S'$.
\end{enumerate} 
\end{conjecture}
The implication $(a)\Rightarrow (b)$ is clear,  while the converse remains open if $\mathrm{char}(k)>0$. In the case of Kummer surfaces, our results provide some evidence of Conjecture \ref{mainconj}.  We shall mention that recently Bragg and Yang have studied the derived isogenies between K3 surfaces over positive characteristic fields and they proved a weaker version of the statement in Conjecture \ref{mainconj} (\cf \cite[Theorem 1.2]{BraggYang2021}).

\subsection*{Organization of the paper.}
We will start with two preliminary sections, in which we include some well-known constructions and facts:
In Section \ref{abelian-surface}, we perform computations for the Brauer group of abelian surfaces using the Kummer construction. This allows us to prove the lifting lemma for twisted abelian surfaces of finite height. 

In Section \ref{Deriso},  we collect the knowledge on derived isogenies between abelian surfaces and their cohomological realizations, which include the motivic realization, the $\bB$-field theory, the twisted Mukai lattices, a filtered Torelli theorem and its relation to the moduli space of twisted sheaves. At the end of 
this section , we follow Bragg and Lieblich's twistor line argument in \cite{BraggLieblich19} to conclude the supersingular case of Theorem \ref{derivediso}.

In Section \ref{section:shiodatrick}, we revise Shioda's work and extend it to rational Hodge isogenies. This is the key ingredient for proving Theorem \ref{TDGC}. Furthermore,  after introducing the admissible $\ell$-adic and $p$-adic bases,  we prove the $\ell$-adic and $p$-adic Shioda's trick for admissible isometries on abelian surfaces.  In an application, we prove the algebraicity of these isometries on abelian surfaces over finitely generated fields. 

Sections \ref{section:derivedchar0} and \ref{section:quasiisogenycharp} are devoted to proving Theorem \ref{TDGC} and Theorem \ref{derivediso}.  Theorem \ref{TDGC} is essentially Theorem \ref{thm:reflectivetwistedderived} and Theorem \ref{thm:twistedisogenychar0}.  The proof of Theorem \ref{derivediso} is much more subtle. We establish the lifting and specialization theorem for prime-to-$p$ derived isogeny. Then we can conclude $(i')\Leftrightarrow (ii')$ from Theorem \ref{TDGC} for abelian surfaces of finite heights.   

\subsection*{Acknowledgement}The authors are grateful for the useful comments by Ziquan Yang. The authors thank the referees for their careful reading and valuable suggestions, which improved this article.

\subsection*{Notations and Conventions}

\subsubsection*{(1)} Throughout this paper, we will use the symbol $k$ to denote a field. If $k$ is a perfect field and $\cha{k}=p>0$, we denote $W \coloneqq W(k)$ for the ring of Witt vectors in $k$, which is equipped with a morphism $\sigma \colon W\rightarrow W$ induced by the Frobenius map on $k$. If $k$ is not perfect, we consider the Cohen ring $W $ with $W/pW =k$. Inside the ring of Witt vectors in a fixed algebraic closure $\bar{k}$ of $k$, we get a fixed Frobenius lift $\sigma \colon W \to W$ of $k$.

\subsubsection*{(2)} Let $X$ be a smooth projective variety over $k$. We denote by $\rH^\bullet_{\et}(X_{\bar{k}},\Zl)$  the  $\ell$-adic \'etale cohomology group of $X_{\bar{k}}$.  The $\Zl$-module $\rH^\bullet_{\et}(X_{\bar{k}},\Zl)$ has been endowed with a canonical $G_k= \Gal(\bar{k}/k)$-action.
We use $\rH^{i}_{\crys}(X/W)$ to denote  the $i$-th crystalline cohomology group of $X$ over the $p$-adic base $W \twoheadrightarrow k$, which is a $W$-module. 

\subsubsection*{(3)}For any abelian group $G$ and an integer $n$, we denote $G[n]$ for the subgroup of $n$ torsions in $G$ and $G\lbrace n \rbrace$ for the union of all $n$-power torsions. For a lattice $L$ in $\ZZ$ or $\QQ$ and an integer $n$, we use $L(n)$ for the lattice twisted by $n$, that is, $L = L(n)$ as $\ZZ$ or $\QQ$-module, but
\[
\langle x, y \rangle_{L(n)} = n \langle x, y \rangle_{L}.
\]
The reader shall not confuse it with the Tate twist.

\subsubsection*{(4)} Let $X$ and $Y$ be abelian surfaces.  Here is a list of all various notions of isogenies between $X$ and $Y$. 
\begin{itemize}
    \item {\bf Isogeny}: a surjective homomorphism $X \to Y$ with finite kernel.
    \item {\bf Quasi-isogeny}:  a $\QQ$-isogeny.
    \item {\bf Prime-to-$\ell$ quasi-isogeny}:  a $\ZZ_{(\ell)}$-isogeny.
    \item {\bf Principal quasi-isogeny}: a quasi-isogeny whose degree is a square. 
    \item {\bf Derived isogeny}: a chain of twisted derived equivalences from $X$ to $Y$.
     \item {\bf Prime-to-$\ell$ derived isogeny}: a derived isogeny whose cohomological realization is prime-to-$\ell$.
\end{itemize}

\section{Twisted abelian surface}
\label{abelian-surface}
In this section, we give some preliminary results in the theory of twisted abelian surfaces, especially those of positive characteristics. Many of them are well-known to experts. %

\subsection{Gerbes on abelian surfaces}\label{subsec:Gerbes}
Let $X$ be a smooth projective variety over a field $k$ and let $\mathscr{X}\rightarrow X$ be a $\mu_n$-gerbe over $X$. This corresponds to a pair $(X,\alpha)$ for some $\alpha\in \rH^2_{\fl}(X,\mu_n)$, where the cohomology group is with respect to the fppf topology. Since $\mu_n$ is commutative, there is a bijection of sets 
\[
 \rH^2_{\fl}(X,\mu_n) \xrightarrow{\sim}\left\{ \text{$\mu_n$-gerbes on $X$} \right\}\!/\simeq 
\]
where $\simeq$ is the $\mu_n$-equivalence defined as in \cite[IV.3.1.1]{Giraud}. We may write $\alpha=[\mathscr{X}]$. For any integer $m$, let $\sX^{(m)}$ be the gerbe corresponding to the cohomological class $m[\sX] \in \rH^2_{\fl}(X,\mu_n)$.

The Kummer exact sequence induces a surjective map 
\begin{equation}\label{braumap}
\rH^2_{\fl}(X,\mu_n)\rightarrow \Br(X)[n] 
\end{equation}
where the right-hand side is the \emph{cohomological Brauer group} $\Br(X) \coloneqq \rH^2_{\et}(X,\mathbb{G}_m)$.   There is an associated $\GG_m$-gerbe on $X$ via the map \eqref{braumap}, denoted by $\sX_{\GG_m}$. 
Let $[\sX_{\GG_m}]$ denote the corresponding class in $\Br(X)[n]$. If $[\sX_{\GG_m}] =0$,  we will call $\sX$ an \emph{essentially-trivial} $\mu_n$-gerbe.

Following \cite[\S 2]{Lieblich07}, one can define the twisted coherent sheaves and the twisted derived category of them in terms of gerbes.

\begin{definition}
Let $\sX \to X$ be a $\mu_n$-gerbe or $\GG_m$-gerbe over $X$. 
Let $\coh^{(m)}(\sX)$ be the abelian category of \emph{$\sX^{(m)}$-twisted coherent sheaves} consists of $m$-fold coherent sheaves on the stack $\sX$.  We define $\rD^{(m)}(\sX)$ as the bounded derived category of $\coh^{(m)}(\sX)$.

As shown in \cite[Proposition 2.1.2.6, Proposition 2.1.3.3]{Lieblich07}, there are natural equivalences \[\coh^{(1)}(\sX) \simeq \coh^{(1)}(\sX_{\Gm}) \simeq \coh(X, [\sX_{\GG_m}])\]
where the last is the abelian category of twisted sheaves defined by  C$\check{\mathrm a}$ld$\check{\mathrm a}$raru \cite{Caldararu}. Throughout this paper, we mainly use Lieblich's terminology. 
\end{definition}

For two $G$-gerbes $\sX \to X$ and $\sY \to Y$, we denote by $\sX \wedge_{i,j} \sY$  the $G$-gerbe on $X \times Y$ given by the image of $G\times G$-gerbe $\sX \times \sY$ under
\[
\rH^2_{\fl}(X\times Y, G \times G) \to \rH^2_{\fl}(X \times Y, G)
\]
induced by the multiplication $G \times G \to G, (g_1,g_2) \mapsto (g^i_1 g^j_2)$. There is an equivalence
\[
\coh^{(1)}(\sX \wedge_{i,j} \sY) \xrightarrow{\sim} \coh^{(i,j)}(\sX \times \sY),
\]
where the right-hand side is the subcategory of $(i,j)$-fold coherent sheaves on $\sX \times \sY$ (\cf~\cite[Corollary 2.3.2]{HLT21}). When $i=j =1$, we simply write $\sX\wedge \sY$ for $\sX \wedge_{1,1} \sY$.

A derived equivalence means a $k$-linear exact equivalence between triangulated categories in the form\[\Phi \colon \D^{(1)}(\sX) \xrightarrow{\sim} \D^{(1)}(\sY).\]If $\Phi$ is of the form\[\Phi^{\cP}(\cE) = \mathbf{R} q_* ( p^* \cE \otimes \cP),\]then we call it a Fourier--Mukai transform with a kernel $\cP \in \D^{(-1,1)}(\sX \times \sY)$ and the projections $p \colon \sX \times \sY \to \sX$, $q \colon \sX \times \sY \to \sY$, and $\sX,\sY$ are called a pair of Fourier--Mukai partners. If these gerbes are (essentially) trivial, then by Orlov's result, any $k$-linear exact equivalence between between bounded derived categories of smooth projective varieties is of this form. 

Similarly to Orlov's theorem, Canonaco and Stellari show that any twisted derived equivalence is also of Fourier--Mukai type.
\begin{proposition}[{\cite{AS07}}]
Any derived equivalence $\D^{(1)}(\sX) \xrightarrow{\sim} \D^{(1)}(\sY)$ can be uniquely (up to isomorphism) as a Fourier--Mukai transform
\[
\Phi^{\cP} \colon \D^{(1)}(\sX) \xrightarrow{\sim} \D^{(1)}(\sY),
\]
whose kernel $\cP$ is a perfect complex in $\D^{(-1,1)}(\sX \times \sY)$.
\end{proposition}

\subsection{Kummer construction}
If $k$ has characteristic $p\neq 2$, there is an associated Kummer surface $\Km(X)$ constructed as follows: 
\begin{equation}
\label{eq:kummerconstruction}
    \begin{tikzcd}
    \widetilde{X} \ar[r,"\widetilde{\sigma}"] \ar[d,"\pi"] & X \ar[d] \\
    \Km(X) \ar[r,"\sigma"] & X/{\iota}
    \end{tikzcd}
\end{equation}
where
\begin{itemize}
    \item $\iota$ is the involution of $X$ given by sending $x$ to $-x$;
    \item $\sigma$ is the crepant resolution of quotient singularities;
    \item $\widetilde{\sigma}$ is the blow-up of $X$ along the closed subscheme $X[2] \subset X$. Its birational inverse is denoted by $\widetilde{\sigma}^{-1}$.
\end{itemize}
Let $E \subset \widetilde{X}$ be the exceptional locus of $\widetilde{\sigma}$. For a  classical cohomology theory  $\rH^{\bullet}(-)$ (such as  Betti, \'etale and crystalline)  with coefficients in $R$, if $2$ is invertible in $R$, we have a canonical decomposition
 \begin{equation}
 \rH^2(\Km(X)) \cong \rH^2(X) \oplus \pi_* \Sigma_X,
 \end{equation}
 where $\Sigma_X$ is the summand in $\rH^2(\widetilde{X})$ generated by irreducible components of $E$. 
 
 Moreover, we have a composition of the sequence of morphisms
\[
(\widetilde{\sigma}^{-1})^{*}: \Br(\widetilde{X}) \to \Br(\widetilde{X} \setminus E) \cong \Br(X\setminus X[2])\cong \Br(X).
\]
Here, the last isomorphism $\Br(X) \to \Br(X\setminus X[2])$ is due to Grothendieck's purity theorem (\cf \cite{Gro68, Ces19}).

\begin{proposition}\label{prop:Kummerbrauer}
When $k= \bar{k}$ and $p \neq 2$, the $(\widetilde{\sigma}^{-1})^*\pi^*$ induces an isomorphism between cohomological Brauer groups
  \begin{equation}\label{Kumiso}
        \Theta \colon \Br(\Km(X))\to \Br(X).
    \end{equation}
In particular, when $X$ is supersingular over $\bar{k}$, then $\Br(X)$ is isomorphic to the additive group $\bar{k}$.
\end{proposition}

\begin{proof}
For torsions of \eqref{Kumiso} whose orders are coprime to $p$, the proof is essentially the same as \cite[Proposition 1.3]{AY09} by the Hochschild--Serre spectral sequence and the fact that $\rH^2(\ZZ/2\ZZ, k^*) =0$ (\cf\cite[Proposition 6.1.10]{WeibelBook}) as $\cha(k) >2$. See also \cite[Lemma 4.1]{Stellari2007} for the case $k = \bbC$.  For $p$-primary torsion part, we have $$\Br(\Km(X))\lbrace p\rbrace \cong \Br(X)^{\iota}\lbrace p \rbrace$$ from the Hochschild--Serre spectral sequence, where $\Br(X)^\iota$ is the $\iota$-invariant  subgroup.   Hence, it suffices to prove that $\iota$ acts trivially on $\Br(X)$. This is well-known to experts and works for any abelian varieties over an algebraically closed field (See the proof of \cite[Lemma 8.1]{Olsson25} for example).

In fact, $\rH^2_{\fl}(X,\mu_p)$ can be $\iota$-equivariantly embedded to $\rH^2_{\dR}(X/k)$ by de Rham--Witt theory (\cf\cite[Proposition 1.2]{Ogus78}). The action of $\iota$ on $\rH^2_{\dR}(X/k) = \wedge^2 \rH^1_{\dR}(X/k)$ is the identity, as its action on $\rH^1_{\dR}(X/k)$ is given by $x \mapsto -x$.
Thus the involution on $\rH^2_{\fl}(X,\mu_p)$ is trivial. 
Then by the exact sequence
\[
0 \to \NS(X) \otimes \mathbb{Z}/p \to \rH^2_{\fl}(X,\mu_p) \to \Br(X)[p] \to 0,
\] we can deduce that $\Br(X)[p]$ is invariant under the involution. Furthermore, for $p^n$-torsions with $n \geq 2$, we can proceed by induction on $n$. Assume that all elements in $\Br(X)[p^d]$ are $\iota$-invariant if $1 \leq d < n$.  By abuse of notation, we still use $\iota$ to denote the induced map $\Br(X)\to \Br(X)$. For $\alpha \in \Br(X)[p^{n}]$,  $p \alpha \in \Br(X)[p^{n-1}]$ is $\iota$-invariant. This gives
\[
p\alpha = \iota(p \alpha) = p \iota(\alpha),
\] which implies $\alpha - \iota(\alpha) \in \Br(X)[p]$.  Applying $\iota$ on $\alpha - \iota(\alpha)$, we can obtain 
\[
\alpha - \iota(\alpha) = \iota(\alpha) - \alpha.
\] 
It implies that $\alpha- \iota(\alpha)$ is also a $2$-torsion element. Since $p$ is coprime to $2$, we can conclude that $\alpha = \iota(\alpha)$.

If $X$ is supersingular, then $\Km(X)$ is also supersingular. We have already known that the Brauer group of a supersingular K3 surface  is isomorphic to $k$  by \cite{Ar74}. Thus, $\Br(X)\cong k$.
\end{proof}

\begin{remark}\label{rmk:absoluteR1}
In the case where $X$ is supersingular,
the method of \cite{Ar74} cannot be applied directly to show that $\Br(X) =k$ as $\rH^1_{\fl}(X, \mu_{p^n})$ is not trivial in general for an abelian surface $X$.
\end{remark}

\subsection{A lifting lemma}
In \cite{Bragg19}, Bragg has shown that a twisted K3 surface can be lifted to characteristic $0$. Though his method cannot be applied directly to twisted abelian surfaces, one can still obtain a lifting result for twisted abelian surfaces via the Kummer construction. The following result will be used frequently in this paper. 
\begin{lemma}\label{lemma:liftingtwistab}
Let $\sX_0\to X_0$ be a $\mathbb{G}_m$-gerbe on an abelian surface $X_0$ over $k=\bar{k}$. Suppose $\mathrm{char}(k)>2$ and $X$ has finite height. Then there exists a complete discrete valuation ring $V$ whose residue field is $k$ and fraction field is $K$ such that 
\begin{itemize}[leftmargin=*]
\item  there is a smooth projective abelian scheme $\sX_V\to X_V$ over $\Spec(V)$  whose special fiber of $\sX_V\to X_V$ is isomorphic to $\sX_0\to X_0$, 
    \item  There is a sequence of isomorphisms
    \[
    \NS(X_{\overline{K}}) \xleftarrow{\sim} \NS(X_V) \xrightarrow{\sim} \NS(X_{0}).
    \]
    Here $\NS(X_V)$ is the group the Cartier divisors on $X_V$ modulo the numerical equivalence over $V$ and the morphisms are given by pull-backs.
\end{itemize}
\end{lemma}
\begin{proof}
The existence of such lifting is ensured by \cite[Theorem 7.10]{Bragg19}, \cite[Lemma 3.9]{LZ21} and Proposition \ref{prop:Kummerbrauer}.
Generally speaking, let $\mathscr{S}_0\to \Km(X_0)$ be the associated twisted Kummer surface via the isomorphism \eqref{Kumiso} in Proposition \ref{prop:Kummerbrauer}. Then \cite[Theorem 7.10]{Bragg19} (by taking the $\Pic(\Km(X_0))$ as the saturated sublattice of itself) asserts that there exists  some discrete valuation ring $V$ and  a projective family of K3 surfaces 
\begin{center}
\begin{tikzcd}
\mathscr{S}_V \ar[r]\ar[rd] & S_V \ar[d]  \\
~ & \Spec (V) 
\end{tikzcd} 
\end{center}
such that  the special fiber is $\mathscr{S}_0\to \Km(X_0)$ and the specialization map of  Néron--Severi lattices $\NS(S_{\overline{K}}) \to \NS(\Km(X_0))$ is an isomorphism, where $K=\mathrm{Frac}(V)$. Now we can apply \cite[Lemma 3.9]{LZ21} to get a lifting $X_V\to \Spec(V)$ of $X$ such that $\Km(X_V)\cong S_V$ over $\Spec(V)$.

Note that we have the isomorphism $\NS(X_V) \cong \NS(X_K)$ since $X_V$ is regular. Consider the following commutative diagram (see \cite[Proposition 3.3]{MaulikPoonen} and its proof):
\begin{equation}\label{eq:lift-ns-ab}
\begin{tikzcd}
     \NS(X_{\bar{K}}) \ar[rr,bend right,"sp"] & \ar[l] \NS(X_K) \cong \NS(X_{V}) \ar[r]& \NS(X_0). 
\end{tikzcd}
\end{equation}
The morphism $\NS(X_V) \to \NS(X_0)$ is injective by \cite[Proposition 3.6]{MaulikPoonen} since $\NS(X_K)$ is torsion-free.
The morphism $\NS(X_K) \to \NS(X_{\overline{K}})$ is a primitive embedding since $\Br(V)=0$. Thus, it is sufficient to see that the specialization map $sp$ is an isomorphism.
The relative Kummer construction $\Km(X_V) \cong S_V$ canonically identifies the $\NS(X_K)$ (resp.~$\NS(X_0)$) as a sublattice of $\NS(S_{\overline{K}})$ (resp.~$\NS(\Km(X_0))$) after dividing $2$ (see \cite[Lemma 7.11]{Ogus78} or \cite[Proposition 3.1]{Shioda79}). Moreover, the identification is compatible under specialization. Then we can conclude it by the isomorphism $\NS(S_{\overline{K}}) \cong \NS(\Km(X_0))$.

To lift the $\GG_m$-gerbe $\sX_0\to X_0$ to $\Spec(V)$, it is equivalent to find a Brauer class in $\Br(X_V)$ such that its restriction to $X_0$ is $[\sX_0]$. Analogous to the proof of Proposition \ref{prop:Kummerbrauer}, there is a canonical map between the cohomological Brauer groups
\[
\Theta =(\widetilde{\sigma}^{-1})^*\pi^* \colon \Br(\Km(X_V))\to \Br(X_V)
\]
as in \eqref{Kumiso}. 
Taking the image $\Theta([\mathscr{S}_V])\in \Br(X_V)$, this is the lifting of $[\sX_0]$ as desired.  
\end{proof}

\subsection{Flat cohomology of abelian surfaces}\label{twistorspace}
Finally, we consider the representability of the flat cohomology of abelian surfaces. Let $f \colon X \to S$ be a flat and proper morphism of algebraic spaces of finite type over $k$. Consider the sheaf of the abelian groups $\R^i\!f_* \mu_p$ on the big fppf site $(\Sch\!/S)_{\fl}$, which can be expressed as the fppf sheafification of
\[
S' \mapsto \rH^i_{\fl}(X_{S'}, \mu_p)
\]
for any $S$-scheme $S'$. The representability of $\R^i\!f_* \mu_p$ is difficult to determine due to the complexity of flat cohomology with $p$-torsion coefficients. In this part, we will prove the representability for abelian surfaces.

\begin{proposition}\label{lemma:R1}
Let $f\colon X \to S$ be an abelian $S$-scheme of relative dimension $2$. Then $\R^1\!f_*\mu_p \cong \widehat{X}[p]$ is a finite flat $S$-group scheme.
\end{proposition}
\begin{proof}
It suffices to check them affine locally on the base. Assume $S$ is an affine scheme of finite type over $k$.
Taking the Stein factorization, we can further assume $f_* \cO_X \cong \cO_S$. Then $f_* \mu_p \cong \mu_p$ also holds universally. Under this assumption, we have an exact sequence of fppf-sheaves by Kummer theory:
\begin{equation}
    0 \to \R^1\!f_* \mu_p \to \R^1\!f_* \GG_m \to \R^1\!f_* \GG_m.
\end{equation}
Since $\R^1\!f_* \GG_m$ computes the relative Picard scheme $\Pic_{X/S}$ and the N\'eron--Severi group of $X$ is torsion-free, we can see
\[
\R^1\! f_* \mu_p \cong \ker\left( \Pic_{X/S} \xrightarrow{\cdot p} \Pic_{X/S} \right) \cong \ker\left( \Pic^0_{X/S} \xrightarrow{\cdot p} \Pic^0_{X/S} \right).
\]

On the other hand, it is well known that $\Pic^0_{X/S}$ is representable by the dual abelian $S$-scheme $\widehat{X}$ (\cf\cite[Corollay 6.8]{GIT}).
Thus, $\R^1\!f_* \mu_p$ is representable by the commutative finite group $S$-scheme $\widehat{X}[p]$.
\end{proof}

\begin{proposition}
\label{prop:representabilityofflat}
Let $f \colon \cX \to S$ be a proper smooth family of abelian surfaces over an algebraic space $S$. Then $\R^2\!f_*\mu_p$ is representable by an algebraic space, which is separated and locally of finite presentation over $S$. 
\end{proposition}
\begin{proof}
This is a consequence of \cite[Theorem 1.8, Example 5.9]{braggolsson21} as $\R^1\!f_*\mu_p$ is representable by Lemma \ref{lemma:R1}.
\end{proof}

\begin{remark}\label{rmk:relativeR1}
The case in which $X \to S = \spec(k)$ is a smooth surface for some field $k$ is claimed by Artin in \cite[Theorem 3.1]{Ar74} without proof. Bragg and Olsson provide a proof (Corollary 1.4 in \cite{braggolsson21}). For relative K3 surfaces, there is a moduli-theoretic proof given by Bragg and Lieblich using the stack of Azumaya algebras (\cf\cite[Theorem 2.1.6]{BraggLieblich19}). Their proof cannot be used directly for relative abelian surfaces as the essential assumption $\R^1\!f_* \mu_p=0$ fails in the fppf site $(\Sch\!/S)_{\fl}$.
\end{remark}
\begin{remark}
 An alternative proof for Proposition \ref{prop:representabilityofflat} is to apply Artin's representability criterion \cite[Theorem 5.3]{Artin69}. The most technical part is to see the separatedness. 
\end{remark}

The following observation is essential in the construction of the twistor space of supersingular abelian or K3 surfaces. 
\begin{corollary}[{\cite[Proposition 2.2.4]{BraggLieblich19}}]\label{cor:connectedcomponent}
Suppose that each geometric fiber of $f \colon \cX \to S$ is supersingular.
The connected components of any geometric fiber of $\R^2\!f_* \mu_p \to S$ are isomorphic to the additive group scheme $\GG_a$. 
\end{corollary}
\begin{proof}
Note that the completion of each geometric fiber of $\R^2\! f_* \mu_p$ at $\bar{s} \in S$, along the identity section, is isomorphic to the formal Brauer group $\widehat{\Br}_{X_{\bar{s}}/k(\bar{s})}$, which is isomorphic to $\widehat{\GG}_a$.
The only smooth connected $p$-torsion group scheme at $k(\bar{s})$ with this property is $\GG_a$.
\end{proof}

\section{Cohomological realizations of derived isogeny}\label{Deriso}

In this section, we provide a summary of the derived isogenies on the cohomology groups of abelian surfaces and introduce the notion of prime-to-$\ell$ derived isogenies. This action can be described in two ways:
\begin{enumerate}
    \item the motivic realization, which provides rational isomorphisms on the cohomology groups;
    \item the realization on the integral twisted Mukai lattices.
\end{enumerate}

Moreover, following the work in \cite{HLT20,LieblichOlsson15}, we extend the filtered Torelli theorem to twisted abelian surfaces over an algebraically closed field $k$ with $\cha(k) \neq2 $. As a corollary, we show that any Fourier--Mukai partner of a twisted abelian surface is isomorphic to a moduli space of stable twisted sheaves (\cf Theorem \ref{thm1}).

\subsection{Motivic realization of derived isogeny on cohomology groups}\label{sec:cohomologyrealization}

It is known that (twisted) derived equivalent smooth projective surfaces over a field $k$ have isomorphic Chow motives (see \cite[\S 2.4]{Huybrechts18} and \cite[\S 1.2]{LV19} for example).
We record these results for the convenience of the reader, focusing on abelian surfaces over $k$ as an example.

For any algebraic surface $X$ over a field $k$,
one may consider idempotent correspondences $\pi^2_{\alg,X}$ and $\pi^2_{\tr,X}$ in $\CH^2(X\times X)_{\QQ}$ defined as
\[
\pi_{\alg,X}^2 \coloneqq \sum_{i=1}^{\rho}\frac{1}{\deg(E_i \cdot E_i)} E_i \times E_i, \quad \pi^2_{\tr,X} = \pi^2_{X} - \pi^2_{\alg,X},
\]
where $\pi^2_{X}$ is the idempotent correspondence given by the Chow--K\"unneth decomposition \eqref{eq:motivicdecomposition} and $E_i$ are divisors generating the N\'eron--Severi group $\NS(X_{k^s})$ such that $E_i \cdot E_i \neq 0$ and $E_i \cdot E_j =0$ for any $i \neq j$.
Consider the decomposition  of $\frh^2(X)$:
\[
\frh^2(X) = \frh_{\alg}^2(X) \oplus \frh_{\tr}^2(X)
\]
given by $\pi^2_{\alg,X}$ and $\pi^2_{\tr,X}$.
It is not hard to see $\frh^2_{\alg}(X)$ is a Tate motive after base change to the separable closure $k^{s}$, whose Chow realization is
\[
\CH_{\QQ}^*(\frh^2_{\alg}(X_{k^s})) \cong \NS(X_{k^s})_{\QQ}.
\]

Let $\Phi^{\cP}\colon \D^{(1)}(\sX) \xrightarrow{\sim} \D^{(1)}(\sY)$ be a  derived equivalence between two twisted abelian surfaces over $k$. Consider the cycle class
\begin{equation}\label{eq:cycle-FM}
    \ch_{\sX^{(-1)} \wedge \sY} (\cP) \cdot \sqrt{\td_{X \times Y}} = \ch_{\sX^{(-1)} \wedge \sY} (\cP) \in \CH^*(X \times Y)_{\QQ}.
\end{equation}
Here $\ch_{\sX^{(-1)} \wedge \sY}(-)$ is the twisted Chern character defined same as in  \eqref{eq:twistedChernCharacter}, this 
 provides an isomorphism
\begin{equation*}
    \frh(X) \xrightarrow{\sim} \frh(Y),
\end{equation*}
which preserves the even-degree parts
\[
\frh^{\even}(-) \coloneqq \bigoplus^{2}_{k=0} \frh^{2k}(-) \cong \bigoplus^{2}_{k=0} \bigwedge^{2k} \frh^{1}(-).
\]
(cf.~\cite[\S\S 1.2.3]{LV19}). For a Weil cohomology theory $\rH$, its cohomological realization
\begin{equation}\label{eq:cohomologicalreal}
    \rH^{\even}(X) \xrightarrow{\sim} \rH^{\even}(Y)
\end{equation}
preserves the Mukai pairing. The cohomological realization \eqref{eq:cohomologicalreal} is not integral in general. We can introduce the prime-to-$\ell$ derived isogeny via the integral cohomological realizations, which will be used in the rest of the paper.

 \begin{definition}\label{def:primetopderived}
Let $\ell$ be a prime and $\cha(k)=p$.  When $\ell\neq p$, a derived isogeny $ \D^b(X) \sim \D^b(Y)$ given by 
\[
\begin{tikzcd}[row sep=0cm, column sep=small]
\D^b(X,\alpha)\ar[r,"\simeq"] &\D^b(X_1, \beta_1) &  &\\
& \D^b(X_1, \alpha_2) \ar[r,"\simeq"]&  \D^b(X_2,\beta_2) &\\
& & \vdots & \\
& & \D^b(X_n, \alpha_{n+1}) \ar[r,"\simeq"] & \D^b(Y,\beta_n)
\end{tikzcd}
\]
is called \textit{prime-to-$\ell$}  if each cohomological realization in the sequence
\[
\widetilde{\varphi}_{\ell} \colon \rH^{\rm even}_{\et}(X_{i-1,\bar{k}},\Ql) \xrightarrow{\sim} \rH^{\rm even}_{\et}(X_{i,\bar{k}},\Ql)
\]
is integral, i.e. $\widetilde{\varphi}_{\ell}\left(\rH^{\rm even}_{\et}(X_{\bar{k}}, \Zl) \right) = \rH^{\rm even}_{\et}(Y_{\bar{k}}, \Zl)$. In the case $\ell=p$,  it is  called \textit{prime-to-$p$} if each $\widetilde{\varphi}_p \colon \rH^{\rm even}_{\crys}(X_{i-1}/K) \xrightarrow{\sim} \rH^{\rm even}_{\crys}(X_{i}/K)$ is integral. 
\end{definition}

\begin{remark}\label{rmk:H2realization}Note that the correspondence \eqref{eq:cycle-FM} does not necessarily preserve the cohomological degrees. However, it admits a modification, that is an isomorphism between degree two parts: 
  Consider the cycle class $\lbrack \Gamma_{\tr} \rbrack  \in \CH^2(X \times Y)_{\QQ},$
given by the codimension two component of \eqref{eq:cycle-FM}. It induces an isomorphism of transcendental motives by a weight argument
\[
\lbrack \Gamma_{\tr}\rbrack_2 \coloneqq\pi^2_{\tr,Y} \circ \lbrack \Gamma_{\tr} \rbrack \circ \pi^2_{\tr,X} \colon \frh^2_{\tr}(X) \xrightarrow{\sim } \frh^2_{\tr}(Y).\]
It extends to an isomorphism $\frh^2(X) \xrightarrow{\sim} \frh^2(Y)$ since their algebraic parts are abstractly isomorphic as $X$ and $Y$ have the same Picard number.
This supports the implication $(v) \Rightarrow (vii)$ in Corollary \ref{cor:TDGC}.
\end{remark}

\subsection{Mukai lattices and $\bB$-fields}\label{subsec:Mukailattice}
Let $k$ be an algebraically closed field with $\cha(k)\neq 2$. Let $X$ be an abelian surface over $k$.   When $k=\CC$, the \emph{Mukai lattice} of $X$ is defined as
\[
\widetilde{\rH}(X,\mathbb{Z}) \coloneqq \rH^0(X,\mathbb{Z}(-1)) \oplus \rH^2(X,\mathbb{Z}) \oplus \rH^4(X,\mathbb{Z}(1))
\]
with the Mukai pairing
\begin{equation}\label{eq:mukaipairing}
\pair{ (r_1, b_1, s_1), (r_2, b_2,s_2)} \coloneqq  b_1\cdot b_2 - r_1 s_2 - r_2s_1,
\end{equation}
and a pure $\ZZ$-Hodge structure of weight $2$. In general, we have the following notion of Mukai lattices \cite[{\S}2]{LieblichOlsson15}. 
Note that the definition there is only for K3 surfaces, but works well for any smooth surface with trivial canonical bundle in fact.
\begin{itemize}[leftmargin=*]
    \item Let $\widetilde{\rN}(X)$ be the \emph{extended N\'eron--Severi lattice} defined as
\[
\widetilde{\rN}(X) \coloneqq \ZZ \oplus \NS(X) \oplus \ZZ,
\]
with Mukai pairing
\[
\langle (r_1, c_1, s_1), (r_2, c_2, s_2) \rangle = c_1\cdot c_2 - r_1 s_2 - r_2 s_1.
\]
The Chow realization of
\[
\frh^0(X)(-1) \oplus \frh^2_{\alg}(X) \oplus \frh^4(X)(1)
\]
can be identified with $\widetilde{\rN}(X)_{\QQ}$.
    \item if $\cha(k)= 0$ or if $\cha(k) =p >0$ and $\ell$ is another prime as usual, then the $\ell$-adic Mukai lattice is defined on the even degrees of integral $\ell$-adic cohomology of $X$ for $\ell$ coprime to $\cha(k)$
    \[
    \rH^0_{\et}(X,\Zl(-1))\oplus \rH^2_{\et}(X, \Zl) \oplus \rH^4_{\et}(X, \Zl(1)),
    \]
    with Mukai pairing defined in a similar formula as \eqref{eq:mukaipairing}
    denoted by $\widetilde{\rH}(X,\Zl)$; or
    \item  if $\cha(k)=p >0$, then the $p$-adic Mukai lattice $\widetilde{\rH}(X,W)$ is defined on the even degrees of crystalline cohomology of $X$ with coefficients in $W(k)$
    \[
    \rH^0_{\crys}(X/W(k))(-1) \oplus \rH^2_{\crys}(X/W(k)) \oplus \rH^4_{\crys}(X/W(k))(1),
    \]
    where the twist $(i)$ is given by changing the Frobenius $F \mapsto p^{-i} F$, and the Mukai pairing is given similarly in the formula \eqref{eq:mukaipairing}.
\end{itemize}
\subsection*{Hodge $\bB$-field}
Assume $k =\CC$. For any $\GG_m$-gerbe $\sX\to X$, one can find a lift $B\in \rH^2(X,\QQ)$ of $[\sX]\in \Br(X)$  from the exponential sequence.  Such $B$ is called a $\bB$-field lift of $\alpha$.  We define the {\it twisted Mukai lattice} of $\sX$ as
\[
\widetilde{\rH}(X,\mathbb{Z};B) \coloneqq \exp(B)\cdot \widetilde{\rH}(X,\mathbb{Z}) \subset \widetilde{\rH}(X,\mathbb{Z}) \otimes_{\mathbb{Z}} \mathbb{Q},
\]
which is isomorphic to $\widetilde{\rH}(X,\ZZ)$. For simplicity of notation, we still use $(r,c,s)$ to denote the vector $\exp(B)(r,c,s)$. There is an induced pure Hodge structure of weight 2 on  $\widetilde{\rH}(X,\ZZ;B)$  given by 
\[
\widetilde{\rH}^{0,2}(X;B) = \exp(B) \widetilde{\rH}^{0,2}(X),
\]
(\cf\cite[Definition 2.3]{HS05}). It is clear that a different choice of such lift $B'$ satisfies $B- B' \in \rH^2(X,\ZZ)$  and thus there is a Hodge isometry
\[\exp(B-B') \colon \widetilde{\rH}(X,\ZZ;B') \xrightarrow{\sim} \widetilde{\rH}(X,\ZZ;B).\]
This means that, up to isomorphisms, $\widetilde{\rH}(X,\ZZ;B)$ is independent of the choice of the $\bB$-field lifting and can also be denoted by $\widetilde{\rH}(\sX,\ZZ)$.

As shown in \cite[Corollary 4.4]{yo06}, for any derived equivalence $\Phi^{\cP} \colon \D^{(1)}(\sX) \xrightarrow{\sim} \D^{(1)}(\sY)$ between two twisted abelian surfaces, the Fourier-Mukai kernel induces a Hodge isometry
\begin{equation}\label{eq:HodgerealizationMukai}
\widetilde{\varphi}= \varphi_{B,B'} \colon \widetilde{\rH}(X,\ZZ;B) \xrightarrow{\sim} \widetilde{\rH}(Y,\ZZ;B')
\end{equation}
for suitable $\mathbf{B}$-field lifts $B,B'$. It provides the cohomological realization as in \eqref{eq:cohomologicalreal} rationally.

\subsection*{$\ell$-adic and crystalline $\bB$-field}  
Let us briefly recall the generalized notions of \textbf{B}-fields in both $\ell$-adic cohomology (\cf\cite[\S 3.2]{LMS07}) and crystalline cohomology (\cf\cite[\S 3]{Bragg20a}) as analogues in Betti cohomology. The complete considerations for the cases $\ell$-adic and $p$-adic are given in \cite[\S 2]{BraggYang2021}, which are applicable to both K3 and abelian surfaces. Therefore, we omit some technical details here. 

For a prime $\ell \neq p$ and $n \in \mathbb{N}$, the Kummer sequence of \'etale sheaves
\begin{equation}\label{eq:etlongexact1}
1 \to \mu_{\ell^n} \to \GG_m \xrightarrow{(\cdot)^{\ell^n}} \GG_m \to 1,
\end{equation}
 induces a long exact sequence
\begin{equation*}
\cdots \Pic(X) \xrightarrow{\cdot l^n} \Pic{X} \to  \rH^2_{\et}(X,\mu_{\ell^n})\rightarrow \Br(X)[\ell^n] \to 0.
\end{equation*}
Taking the inverse limit $\varprojlim_{n}$, we get a map
\[
\pi_{\ell} \colon \rH^2_{\et}(X,\Zl(1)) = \varprojlim_{n} \rH^2_{\et}(X,\mu_{\ell^n}) \to \rH^2_{\et}(X,\mu_{\ell^n}) \twoheadrightarrow \Br(X)[\ell^n].
\]
\begin{lemma}\label{lemma:ellsurjective}
The map $\pi_{\ell}$ is surjective.
\end{lemma}
\begin{proof}
We have a short exact sequence (\cf\cite[Chap.V, Lemma 1.11]{MilneEtaleBook})
\[
0 \to \rH^2_{\et}(X,\Zl(1)) / \ell^n  \to \rH^2_{\et}(X,\mu_{\ell^n}) \to \rH^3_{\et}(X,\Zl(1))[\ell^n]\to 0.
\]
Since $\rH^3_{\et}(X,\Zl(1))$ is torsion-free for any abelian surface $X$,  we have an isomorphism
\[
\rH_{\et}^2(X,\Zl(1))/\ell^n \cong \rH^2_{\et}(X,\mu_{\ell^n}).
\]
Therefore, the reduction morphism $\rH^2_{\et}(X,\Zl(1)) \to \rH^2_{\et}(X, \mu_{\ell^n})$ can be identified with
\[
\rH^2_{\et}(X,\Zl(1))\twoheadrightarrow \rH^2_{\et}(X,\Zl(1))/\ell^n,
\]
which is surjective.
The assertion then follows from it. 
\end{proof}
For any $\alpha \in \Br(X)[\ell^n]$ such that $\ell \neq p$, let $B_{\ell}(\alpha) \coloneqq \pi^{-1}_{\ell}(\alpha)$, which is nonempty by Lemma \ref{lemma:ellsurjective}. 

For Brauer class $\alpha \in \Br(X)[p^n]$, we need the following commutative diagram via the de Rham--Witt theory (\cf\cite[I.3.2, II.5.1,  Th\'eor\`eme 5.14]{Il79})
\begin{equation}\label{eq:deRhamWittdiag}
\begin{tikzcd}
0 \ar[r] &\rH^2(X,\Zp(1)) \ar[r] \ar[d] & \rH_{\crys}^2(X/W) \ar[d,"p_n \coloneqq(\otimes W_n)"] \ar[r,"p-F"] & \rH_{\crys}^2(X/W)\\
 &\rH^2_{\fl}(X,\mu_{p^n})\ar[r,"d\log"] & \rH^2_{\crys}(X/W_n) &
\end{tikzcd}
\end{equation}
where $\rH^2(X,\Zp(1)) \coloneqq \varprojlim_n \rH^2_{\fl}(X,\mu_{p^n})$. The map $d\log$ is known to be injective by flat duality (\cf\cite[Proposition 1.2]{Ogus78}).  Since the crystalline cohomology groups of an abelian surface are torsion-free, the mod $p^n$ reduction map $p_n$  is surjective.
Consider the canonical surjective map
\[
\pi_p \colon \rH^2_{\fl}(X,\mu_{p^n}) \twoheadrightarrow \Br(X)[p^n],
\]
induced by the Kummer sequence.
We set 
\[
B_{p}(\alpha)\coloneqq\set*{ b \in \rH^2_{\crys}(X/W) \given  \parbox{12em}{  $p_n(b) = d\log(t)$ for some $t \in \rH^2_{\fl}(X, \mu_{p^n})$ such that  $\pi_p(t)=\alpha$} }.
\]

Following \cite[Definition 2.16, 2.17]{BraggYang2021}, we can introduce the (mixed) $\bB$-fields for twisted abelian surfaces.
\begin{definition} Let $\sX\to X$ be a $\mu_n$-gerbe and $[\sX_{\GG_m}]\in\Br(X)[n]$.
\begin{itemize}
    \item If $ n=\ell^t$ for some prime $\ell$, an {\it $\ell$-adic \textbf{B}-field lift} of $\sX\to X$ is an element $B=\frac{b}{\ell^t}$, where $b\in B_\ell([\sX_{\GG_m}])$. When $\ell=p$, it is also called a {\it crystalline \textbf{B}-field}  lift.

    \item In general,  a mixed $\bB$-field lift of $\sX\to X$ is a collection $B=\{B_{\ell}\}$ consisting of  a choice of an $\ell$-adic $\bB$-field lift $B_\ell$ of $[\sX_{\GG_m}^{(n\ell^{-t_\ell})}]$ for all prime factors $\ell\mid n$, where  $t_\ell$ is the $\ell$-adic valuation of $n$.
\end{itemize}
 
\end{definition}
\begin{remark}
Not all elements in $\rH^2_{\crys}(X/W)[\frac{1}{p}]$ are crystalline \textbf{B}-fields since  the map $d\log$ is not surjective. From the first row in the diagram \eqref{eq:deRhamWittdiag}, we can see $B \in \rH^2_{\crys}(X/W)[\frac{1}{p}]$ is a \textbf{B}-field lift of some Brauer class if and only if $F(B) = pB$.
\end{remark}

\subsection{Twisted Mukai lattice over arbitrary fields}
Let $\pi \colon \sX \to X$ be a $\mu_n$-gerbe and $\ord([\sX_{\Gm}])=n$, $B= \set{B_{\ell}}$ a mixed $\bB$-field lift of $[\sX_{\Gm}]$.  We define the $\ell$-adic twisted Mukai lattice as 
\begin{equation}\label{Mlat}
\widetilde{\rH}(X,B_{\ell})=\begin{cases}
\exp(B_{\ell}) \widetilde{\rH}(X,\Zl) & \text{if} ~\ell \neq p\\ ~\\ 
\exp(B_{\ell})\widetilde{\rH}(X,W)& \text{if}~\ell=p
\end{cases}
\end{equation}
endowed with the Mukai pairing \eqref{eq:mukaipairing}, where $\exp(B_{\ell}) = 1 + B_{\ell} + \frac{B^2_{\ell}}{2}$.  

Up to isomorphisms, the twisted Mukai lattice $\widetilde{\rH}(X,B_{\ell})$ is independent of the choice of the $\bB$-field lift. We may use $\widetilde{\rH}(\sX,\Zl)$ or $\widetilde{\rH}(\sX,W)$ to denote the twisted Mukai lattices to highlight the coefficients, irrespective of the choice of the $\bB$-field lift. 

\begin{definition}

 Let $\rK_0^{(1)}(\sX)$  be the Grothendieck group of $\coh^{(1)}(\sX)$. 
The map of \emph{twisted Chern character} is the unique additive group homomorphism
\[
\ch_{\sX} \colon \rK^{(1)}_0(\sX) \to \widetilde{\rN}(X)_{\QQ}
\]
such that for any locally-free $\sX$-twisted sheaf $\cE$ on $\sX$ with positive rank
\begin{equation}\label{eq:twistedChernCharacter}
    \ch_{\sX}(\cE) = \sqrt[n]{\pi_*(\cE^{\otimes n})} \in \widetilde{\rN}(X)_{\QQ},
\end{equation}
where $\sqrt[n]{-}$ means a choice of $n$-roots such that the $0$-codimension component of $\ch_{\sX}(\cE)$ is equal to $\rank~\cE$.

Denote by $\widetilde{\rN}(\sX)$ the image of $\rK_0^{(1)}(\sX)$ in $\widetilde{\rN}(X)_{\QQ}$ under the twisted Chern character map $\ch_{\sX}$, called \emph{extended twisted N\'eron-Severi lattice}.   For $\cE\in \rD^{(1)}(\sX)$,  we define $v(\cE)=\ch_{\sX}([\cE])\in \widetilde{\rN}(\sX)$ to be the Mukai vector of $\cE$. 

\end{definition}

One can also define the map of twisted Chern character to cohomological twisted Mukai lattice
\[
\ch_{B} \colon \rK_0^{(1)}(\sX) \to \widetilde{\rH}(X,B_{\ell}),
\]
see \cite[\S 3.3]{LMS07} and \cite[Appendix A3]{Bragg20a} for $\ell$-adic and crystalline cases, respectively.  For any mixed $\bB$-field lift $B$ of $[\sX_{\Gm}]$, the twisted Chern character $\ch_{B}$ factors through $\widetilde{\rN}(\sX)$:
\[
\begin{tikzcd}[column sep= small]
\rK_0^{(1)}(\sX) \ar[rr,"\ch_{B_{\ell}}"] \ar[rd,"\ch_{\sX}"']& & \widetilde{\rH}(X,B_{\ell}) \\
&\widetilde{\rN}(\sX) \ar[ru, "\exp(B_{\ell})\cl_{\rH}"']
\end{tikzcd}
\]
where $\cl_{\rH}$ is the cycle class map to the cohomology theory $\rH(-)$. %
The following result is essentially proved in \cite{BraggYang2021}.

\begin{proposition}\label{prop:extendedNeronSeveri}
Let $B$ be a mixed $\bB$-field lift of $[\sX_{\Gm}] \in \Br(X)$.
Then 
\[
\widetilde{\rN}(\sX)\cong \bigcap_{\ell}\left(\widetilde{\rN}(X) \otimes \ZZ[\frac{1}{\ell}] \cap \widetilde{\rH}(X,B_{\ell}) \right).
\]
where the intersection $\widetilde{\rN}(X) \otimes \ZZ[\frac{1}{\ell}] \cap \widetilde{\rH}(X,B_{\ell})$ is taken in $\widetilde{\rN}(X)\otimes \QQ_\ell$ and the intersection  $\bigcap_{\ell}$ is taken in $\widetilde{\rN}(\widetilde{X})\otimes \QQ$. 
In particular, the lattice $\widetilde{\rN}(\sX)$ only depends on  the associated $\GG_m$-gerbe $\sX_{\GG_m}$, up to a lattice isomorphism.
\end{proposition}
\begin{proof}
This is \cite[Proposition 3.5]{BraggYang2021}.
\end{proof}

Similarly, one can  define the relative extended twisted Mukai lattice on smooth projective families of twisted abelian surfaces.

%
%
%
%
%
%
%
%
%
%
%
%
%
%
%
%
%
%

%

%
%

%

\subsection{A filtered Torelli Theorem}
In \cite{LieblichOlsson15,LieblichOlsson17}, Lieblich and Olsson have introduced the filtered derived equivalence and demonstrated that K3 surfaces with such equivalence are isomorphic. We will present an analogous result for (twisted) abelian surfaces. The proof is simpler than for K3 surfaces, as the bounded derived category of a (twisted) abelian surface corresponds to a generic K3 category \cite{Huybrechts2008}.

Let $\sX\to X$ be a $\mu_n$-gerbe.  The rational numerical Chow ring $\CH^*_{\num}(\sX)_{\QQ}\cong \CH^*_{\num}(X)_{\QQ}$ is equipped with a codimension filtration
\[
\fil^i \CH^*_{\num}(\sX)_{\QQ} \coloneqq \bigoplus_{k \geq i}\CH^k_{\num}(\sX)_{\QQ}.
\]
As $X$ is a surface, we have a natural identification $\widetilde{\rN}(\sX)_{\QQ} \cong \CH^*_{\num}(\sX)_{\QQ}$.

\begin{definition}
    Let $\Phi^{\cP} \colon \rD^{(1)}(\sX) \to \rD^{(1)}(\sY)$ be a Fourier--Mukai transform. The derived equivalence $\Phi^{\cP}$ is called \emph{filtered} if its induced isomorphism $\Phi^{\cP}_{\CH} \colon \widetilde{\rN}(\sX) \xrightarrow{\sim} \widetilde{\rN}(\sY)$ preserves the induced codimension filtrations.
\end{definition} Since the isomorphism $\widetilde{\rN}(\sX) \xrightarrow{\sim} \widetilde{\rN}(\sY)$ preserves the Mukai pairing, it is not hard to see that $\Phi^{\cP}$ is filtered if and only if it sends the Mukai vector $(0,0,1)$ to $(0,0,\pm 1)$. At the cohomological level, the codimension filtration on $\widetilde{\rH}(X)[\frac{1}{\ell}]$ (the prime $\ell$ depends on the choice of $\ell$-adic or crystalline twisted Mukai lattice) is given by $F^i = \oplus_{r \geq i} \rH^{2r}(X)[\frac{1}{\ell}]$. The filtration on $\widetilde{\rH}(\sX,\ZZ_{\ell})$ is defined by
\[
F^i \widetilde{\rH}(\sX,\ZZ_{\ell}) = \widetilde{\rH}(\sX,\ZZ_{\ell}) \cap F^i\widetilde{\rH}(X,\ZZ_{\ell})[\frac{1}{\ell}].
\]
By choosing a \textbf{B}-field lift $B_\ell$, a direct computation shows that the graded pieces of $F^\bullet$ are
\begin{equation}\label{eq:gradedpieces}
\begin{aligned}
&\Gr^0_F \widetilde{\rH}(\sX,\ZZ_{\ell}) = \set*{\left(r, rB_\ell, \frac{rB^2_\ell}{2}\right) \given r \in \rH^0(X,\ZZ_{\ell}(-1))},\\
&\Gr^1_F \widetilde{\rH}(\sX,\ZZ_{\ell})= \set*{
(0, b, b \cdot B_\ell) \given b \in \rH^2(X,\ZZ_{\ell})} \cong \rH^2(X,\ZZ_{\ell}),\\
& \Gr_F^2 \widetilde{\rH}(\sX,\ZZ_{\ell}) = \set*{(0,0, s)\given s \in \rH^4(X,\ZZ_{\ell}(1)) } \cong \rH^4(X,\ZZ_{\ell}(1)).
\end{aligned}
\end{equation}

\begin{lemma}\label{lemma:filteredcofiltered}
A Fourier--Mukai transform $\Phi^{\cP}\colon \rD^{(1)}(\sX)\to \D^{(1)}(\sY)$ is filtered if and only if its cohomological realization is filtered for any \textbf{B}-field liftings.
\end{lemma}
\begin{proof}
A Fourier--Mukai transform that is filtered implies that it is cohomologically filtered. This is because the map \[
\exp(B_\ell) \cdot \cl_H \colon \widetilde{\rN}(\sX) \to \widetilde{\rH}(\sX,\ZZ_\ell)
\]
preserves the filtrations for any \textbf{B}-field lift $B$ of $[\sX_{\GG_m}]$.

For the converse, just notice that $\Phi^{\cP}$ is filtered if and only if the induced map $\Phi^{\cP}_{\CH}$ takes the vector $(0,0,1)$ to $(0,0,\pm 1)$.  As $\Phi^{\cP}$ is cohomologically filtered for $B$, the cohomological realization of $\Phi^{\cP}$ preserves the graded piece $\Gr_F^2$ in \eqref{eq:gradedpieces}.  This implies that $\Phi^{\cP}_{\CH}$  takes $(0,0,1)$ to $(0,0,\pm 1)$.
\end{proof}

\begin{proposition}[filtered Torelli theorem for twisted abelian surfaces]
\label{thmfilter}
Suppose $k= \bar{k}$ is such that $\cha(k) \neq 2$.
    Let $\sX \to X$ and $\sY \to Y$ be $\mu_n$-gerbes on abelian surfaces. The following statements are equivalent.
    \begin{enumerate}
        \item There is an isomorphism between the associated $\GG_m$-gerbes $\sX_{\mathbb{G}_m}$ and $\sY_{\GG_m}$.
        \item There is a filtered Fourier--Mukai transform $\Phi^{\cP}:\rD^{(1)}(\sX)\to \rD^{(1)}(\sY)$.
    \end{enumerate}
\end{proposition}

\begin{proof} For untwisted case, i.e. $\sX=X$ and $\sY=Y$,  this is exactly  \cite[Proposition 3.1]{HLT20}. Here we extend it to the twisted case. As one direction is obvious, it suffices to show that  (2) can  imply (1).

Firstly, we claim that all semi-rigid objects in $\D^{(1)}(\sY)$ are in $\coh^{(1)}(\sY)$ up to shift. According to Remark 3.13 in \cite{Huybrechts2008}, it is sufficient to show that there are no stable spherical sheaves in $\coh(\sY^{(1)})$. 
If $\cE$ is a spherical $\sY^{(1)}$-twisted sheaf with $\rank~\cE=0$, then $c_1(\cE)^2 =- \chi (\cE, \cE) = -2$, which is impossible for the abelian surface.
Suppose that there is a stable spherical $\sY$-twisted sheaf $\cE$ with Mukai vector $v= (r,c,s)$ such that $r >0$. Choose a polarization $H\in \Pic(Y)$ so that $\cE$ is $H$-semistable. Let $M_H(\sY,v)$ be the moduli space of $H$-semistable $\sY$-twisted sheaves on $Y$.  Then $M_H(\sY,v)$ is non-empty. Consider the determinant morphism to the Picard stack of invertible $\sY^{(r)}$-twisted sheaves
\[
\mathbf{det} \colon \sM_H(\sY,v) \to \Pic(\sY^{(r)}).
\]
For any $\cL \in \Pic^0(Y)$ and $\cE \in \sM_H(\sY,v)$, the tensor product $\cE \otimes \cL$ is still a stable $\sY$-twisted sheaf with the Mukai vector $v$. Thus, the map $\mathbf{det}$ dominates the component of $\Pic(\sY^{(r)})$ containing $\det(\cE)$, which is of dimension $2$. Therefore, the deformation theory of twisted coherent sheaf implies
\[
\dim_k \Ext^1(\cE, \cE) \geq \dim \sM_H(\sY,v) \geq 2,
\]
contradicting the assumption that $\cE$ is spherical.

Let $\Phi^{\cP} \colon \rD^b(\sX^{(1)}) \to \rD^b(\sY^{(1)})$ be a Fourier--Mukai transform. For a closed point $x \in X$, denote
\[
  \cP_{x} \coloneqq \Phi^{\cP}(k(x)) =\cP|_{\{x\}\times \sY},
\]
by image of the skyscraper sheaf $k(x)$. Since $k(x)$ is semi-rigid, $\cP_x$ is also semi-rigid. The previous discussion implies that there is an integer $m$ such that $\mathcal{H}^i(\cP_x)=0$ for any $i \neq m$ and closed point $x \in \cY$. Therefore, there is a $\sX^{(-1)} \wedge \sY$-twisted sheaf $\cE\in \coh(\sX^{(-1)} \times \sY)$ such that $\cP \cong \cE[m]$. 

Suppose $\Phi^{\cP}$ is filtered. Composing it with the shift functor $\cF \mapsto \cF[1]$ if necessary, we may assume that the cohomological realization of $\Phi^{\cP}$ sends $(0,0,1)$ to $(0,0,1)$. In this case, $\cE_x$ is just a skyscraper sheaf on $\{x\}\times Y$. The same argument as in \cite[Corollary 5.3]{AS07} or \cite[Corollary 5.22, 5.23]{FMtransform} shows that there is an isomorphism $f \colon X \to Y$ such that $f^*([\sY_{\mathbb{G}_m}]) = [\sX_{\mathbb{G}_m}]$. \end{proof}

\subsection{Twisted FM partners via  moduli space of twisted sheaves}

In the rest of this section, we will always assume that $k = \bar{k}$ and $\cha(k) = p \neq 2$.  Let $\sX\to X$ be a twisted abelian surface over $k$. 

\begin{definition}[{\cite[Definition 0.1]{Yo01}}]\label{def:positive}
    Let $v=(r,c,s) \in \widetilde{\rN}(\sX)$ be a primitive Mukai vector such that $v^2=0$. If one of the following holds
\begin{enumerate}
    \item  $r>0$.
    \item $r=0$, $c$ is effective and $s \neq 0$.
    \item $r=c=0$ and $s>0$.
\end{enumerate}
then $v$ is called \emph{positive}.
\end{definition}

We denote by $\srM_H(\sX,v)$  the moduli stack of $H$-semistable $\sX$-twisted sheaves with the Mukai vector $v \in \widetilde{\rN}(\sX)$, where $H$ is a $v$-generic ample divisor on $X$. Here, we record a well-known non-emptiness criterion for $\sM_H(\sX,v)$ when $X$ is not supersingular. We will extend this result to the supersingular case in Proposition \ref{prop:SupersingularModuli}, using the theory of supersingular twistor space.

\begin{proposition}[Minamide--Yanagida--Yoshioka, Bragg--Lieblich]\label{non-emptyness}
\label{criteria}Suppose $X$ is an abelian surface over $k$ that is not supersingular.
If $v$ is positive with $v^2 =0$, then for any $v$-generic polarization $H$, the coarse moduli space $M_H(\sX,v)$ is an abelian surface, and the moduli stack $\srM_H(\sX, v)$ is a $\GG_m$-gerbe on $M_H(\sX,v)$. 
\end{proposition}
\begin{proof}

For the case where $\cha(k) = 0$, Yoshioka has proved this result in \cite[Theorem 3.16]{yo06}. 

When $\cha(k) = p > 2$, the nonemptiness can be seen through a lifting argument, as shown in \cite[Proposition 4.1.20]{BraggLieblich19} and \cite[Proposition A.2.1]{MMY11}. Since $X$ is of finite height when $\cha(k) =p >0$, Lemma \ref{lemma:liftingtwistab}, implies exists a DVR $V$ with residue field $k$ and a projective lifting $\sX_V \to X_V$ of $\sX \to X$ over $\Spec(V)$,  together with an extension $v_V \in \widetilde{\rN}(\sX_V)$ and a polarization $H_V\in \NS(X_V)$ such that $H_V|_{\Spec(k)} = H$.  Consider the relative moduli space of twisted sheaves $\sM_{H_V}(\sX_V, v_V)$ over $\Spec(V)$. Its (geometric) generic fiber is a moduli space of twisted sheaves with positive Mukai vector in characteristic zero, which is nonempty by Yoshioka's result. Thus its special fiber, which is isomorphic $\sM_H(\sX,v)$, is also nonempty by Langton's semi-stable reduction theorem.
\end{proof}

The following is an extension of \cite[Theorem 1.2]{HLT20}.
\begin{theorem}\label{thm1}
Assume $k = \bar{k}$ with $\cha(k) \neq 2$. Let $\sX\to X$ and $\sY\to Y$ be $\GG_m$-gerbes over an abelian surface defined over $k$. Then $\rD^{(1)}(\sX)\simeq \rD^{(1)}(\sY)$ if and only if $\sY^{(-1)} \to Y$ is isomorphic to the moduli stack $\sM_{H}(\sX,v)\to M_H(\sX,v)$ for some $v\in\widetilde{\rN}(\sX)$ and $v$-generic polarization $H$.
\end{theorem}
\begin{proof}
For the "if" part, just note that the universal family of twisted sheaves on $\srM_H(\sX,v) \times \sX$ induces a derived equivalence. 

For the other direction, suppose $\rD^{(1)}(\sX)\simeq \rD^{(1)}(\sY)$ are equivalent.
We let \[ \Phi^{\cP} \colon \D^{(1)}(\sY) \to \D^{(1)}(\sX) \] be a Fourier--Mukai transform. 
Let $v\in \widetilde{\rN}(\sX)$ be the image of $(0,0,1)\in \widetilde{\rN}(\sY)$ under $\Phi^{\cP}$. Up to a shift, we can assume that $v$ is a positive vector. By Proposition \ref{criteria},  $M_H(\sX,v)$ is an abelian surface and $\sM_H(\sX,v)\to M_H(\sX,v)$ is a $\GG_m$-gerbe over it. 

Let $\cE$ be a universal $\sX$-twisted sheaf on $\sM_{H}(\sX,v) \times \sX$, which is a $(1,1)$-fold twisted sheaf and induces a derived equivalence
\[
    \Phi^{\cE} \colon \D^{(-1)}(\sM_{H}(\sX,v)) \to \D^{(1)}(\sX),
\]
whose cohomological realization maps the Mukai vector $(0,0,1)$ to $v$. Composing $\Phi^{\cE}$ with the derived equivalence
\[(\Phi^{\cP})^{-1} \simeq \Phi^{\cP^{\vee}}[2] \colon \D^{(1)}(\sX) \to \D^{(1)}(\sY),\] we obtain a filtered derived equivalence from \(\srM_H(\sX,v)^{(-1)}\) to \(\sY\), which induces an isomorphism from \(\srM_H(\sX,v)^{(-1)}\) to \(\sY\) by Theorem \ref{thmfilter}.
\end{proof}

\subsection{Supersingular twisted abelian surfaces}

Finally, we discuss the case of supersingular twisted abelian surfaces. In this part, we extend the construction the supersingular twistor space as \cite{BraggLieblich19} via the Ogus's crystalline Torelli theorem for supersingular abelian (see \cite[\S 2]{Ogus78}).

\begin{definition}Let $p$ be a prime $\neq 2$. Let $\Lambda$ be an indefinite  $p$-elementary even lattice, i.e., 
\begin{itemize}
    \item  $\disc(\Lambda \otimes \QQ)=-1$,
    \item  $\Lambda^{\vee}/\Lambda$ is $p$-torsion.
\end{itemize}  Then $\lvert \Lambda^{\vee}/\Lambda \rvert = p^{2\sigma_0(\Lambda)}$ for $1\leq \sigma_0(\Lambda)\leq \frac{n}{2}$ and the integer $\sigma_0(\Lambda)$ is called the \emph{Artin invariant} of $\Lambda$.   We define $M_{\Lambda}$ to be  Ogus' \emph{moduli space of  characteristic subspaces} of $ p\Lambda^{\vee}/p\Lambda$.
\end{definition}

When $\Lambda$ has the signature $(1, n-1), n \geq 2$, as shown in \cite[Section 1]{RS81}, $\Lambda$ is uniquely determined by its Artin invariant.  When $n=6$, we may call it \emph{supersingular abelian surface lattice}. This is because for every supersingular abelian surface $X$, its Néron-Severi lattice $\NS(X)$ is a supersingular abelian surface lattice (cf.~\cite[(1.6)]{Ogus83}.

From now on, let us assume that $\Lambda$ is a supersingular abelian surface lattice. Denote $\sigma_0$ for the Artin invariant $\sigma_0(\Lambda)$ for simplicity. We set $$\widetilde{\Lambda}= \Lambda \oplus U(p),$$ where $U(p)$ is the twisted hyperbolic plane generated by the vectors $e$ and $f$ such that $e^2= f^2=0$ and $e\cdot f=-p$. Let $M_{\widetilde{\Lambda}}^{\langle e \rangle} \subseteq M_{\widetilde{\Lambda}}$ be the moduli space of characteristic subspaces of $p \widetilde{\Lambda}^{\vee} / p \widetilde{\Lambda}$ that do not contain $e$. 

\begin{proposition}[{\cite[\S 3]{BraggLieblich19}}]
The moduli stack $M_{\widetilde{\Lambda}}^{\langle e \rangle}$ and $M_\Lambda$ are representable by schemes over $\mathbb{F}_p$, which are smooth of dimensions $\sigma_0$ and $\sigma_0 -1$, respectively. Moreover, there is a smooth morphism \[
\pi_{e} \colon M_{\widetilde{\Lambda}}^{\langle e \rangle} \to M_{\Lambda}.
\] whose fiber at a closed point is isomorphic to a group scheme with connected components $\mathbb{A}^1$.
\end{proposition}
\begin{proof}
The first assertion is given in \cite[Proposition 4.6]{Ogus83}.
Let us sketch the construction of $\pi_e$. 
Given any $\widetilde{\cK} \in M_{\widetilde{\Lambda}}^{\langle e \rangle}(T)$ over an $\mathbb{F}_p$-scheme $T$, a characteristic subspace $$\cK\subseteq ( p\Lambda^{\vee}/p\Lambda) \otimes \cO_T$$ can be formed as the image of $\widetilde{\cK} \cap (e^{\bot} \otimes \cO_T)$ in $ (e^{\bot}/e) \otimes \cO_T$ (see \cite[Lemma 3.1.9]{BraggLieblich19}). Consequently, the map $\widetilde{\cK} \mapsto \cK$ defines a morphism
\[ \pi_{e} \colon M_{\widetilde{\Lambda}}^{\langle e \rangle} \to M_{\Lambda}. \]
The rest of the assertion is a consequence of \cite[Lemma 3.1.15 ]{BraggLieblich19}.
\end{proof}

\begin{definition}
      The twistor line in $M_{\widetilde{\Lambda}}\otimes_{\FF_p} k$ is an affine line $\AA^1 \subset M_{\widetilde{\Lambda}}\otimes_{\FF_p} k$ that is a connected component of a fiber of $\pi_e$ over a $k$-point of $ M_{\Lambda}(k)$ for some isotropic vector $e \in \widetilde{\Lambda}$.
\end{definition}

The moduli functor $S_{\Lambda}$ of $\Lambda$-marked supersingular abelian surfaces is representable by a locally separated and smooth algebraic space of dimension $\sigma_0-1$ over $k$ by the crystalline Torelli theorem \cite[Theorem 7.3]{Ogus78} together with the argument in \cite[Theorem 2.7]{Ogus83}. 
Consider the universal family of supersingular abelian surfaces
\[
u \colon \cX \to S_{\Lambda},
\] that is smooth with relative dimension $2$.
By Proposition \ref{prop:representabilityofflat}, the higher direct image $\R^2\!u_* \mu_p $ is representable by an algebraic group space over $S_{\Lambda}$, denoted by
\[
\pi \colon \mathscr{S}_{\Lambda} \to S_{\Lambda}.
\] The connected component of the identity $\mathscr{S}_{\Lambda}^{o} \subset \mathscr{S}_{\Lambda}$ parameterizes the $\mu_p$-gerbes which are not essentially-trivial except the identity, on each $\Lambda$-marked supersingular abelian surface in $S_{\Lambda}(k)$.  Then there are (twisted) period morphisms following the approach in \cite[\S 3]{Ogus83}.

\begin{proposition}\label{prop:twistedperiodmorphism}
There are (twisted) period morphisms $\rho \colon S_{\Lambda}\to \overline{M}_{\Lambda}\coloneqq M_{\Lambda} \otimes_{\FF_p} k$ and $\widetilde{\rho} \colon \mathscr{S}_{\Lambda}^o \to \overline{M}_{\widetilde{\Lambda}}^{\langle e \rangle} \coloneqq M_{\widetilde{\Lambda}}^{\langle e \rangle}\otimes_{\FF_p}k$ such that the following commutative diagram is Cartesian
\begin{equation}\label{eq:period-diagram}
\begin{tikzcd}
\mathscr{S}_{\Lambda}^o \ar[r,"\pi|_{\mathscr{S}_{\Lambda}^o}"] \ar[d,"\widetilde{\rho}_\Lambda"] & S_{\Lambda} \ar[d, "\rho_\Lambda"] \\
\overline{M}_{\widetilde{\Lambda}}^{\langle e \rangle}\ar[r,"\pi_e"] & \overline{M}_{\Lambda}
\end{tikzcd}
\end{equation}
Moreover, $\rho$ and $\widetilde{\rho}$ are \'etale surjective when $p>2$.  
\end{proposition}

\begin{proof}
This was proved by Bragg and Lieblich in the case of supersingular K3 surfaces (cf.~ \cite[\S 3 and \S 5]{BraggLieblich19}. But everything works for supersingular abelian surfaces as well. We shall mention that one can also use the Kummer construction to deduce the statement from the K3 case.

For the ease of the reader, let us briefly sketch the construction of $\widetilde{\rho}_{\Lambda}$ and $\rho_\Lambda$. Let $(X,\eta)$ be a $\Lambda$-marked supersingular abelian surface. The K3-crystal $\rH^2_{\crys}(X/W)$ determines a characteristic subspace
\[ \cK_{\rH^2(X)}\coloneqq  \ker(\NS(X)\otimes k \to \rH^2_{\crys}(X/W) \otimes k).\]
Then $\rho_{\Lambda}(X,\eta)$ is the characteristic subspace $\eta^{-1}(\cK_{\rH^2(X)})$ in $(p\Lambda^{\vee}/p\Lambda) \otimes_{\FF_p}k$. Suppose $\sX \to X$ is a $\mu_p$-gerbe. We define
\[
\cK_{\widetilde{\rH}(\sX)} \coloneqq  \ker( \widetilde{\rN}(\sX) \otimes k \to \widetilde{\rH}(\sX,W) \otimes_W k) \subset \frac{p\widetilde{\rN}(\sX)^{\vee}}{p \widetilde{\rN}(\sX)} \otimes k
\]
as the strictly characteristic subspace of $\widetilde{\rH}(\sX,W)$. Note that there is an extended map of K3 crystals:
\[
\widetilde{\Lambda} \otimes \ZZ_p \xrightarrow{\widetilde{\eta}}\widetilde{\rN}(\sX)\otimes \ZZ_p \to \widetilde{\rH}(\sX,W) 
\]
where $\widetilde{\eta}$ is given by
\[
e \mapsto (0,0,1) \quad c \mapsto \eta(c)  \quad f \mapsto (p,0,0)
\]
for any $c \in \Lambda$.  Then $\widetilde{\rho}_{\Lambda}(\sX,\eta) = \widetilde{\eta}^{-1} (\cK_{\widetilde{\rH}(\sX)})$ is the characteristic subspace of $(p\widetilde{\Lambda}^{\vee}/p\widetilde{\Lambda}) \otimes k$.
\end{proof}

\begin{remark}\label{rmk:twistedPeriod}
In one view, the moduli space $M_\Lambda$ is a crystalline analog of the classical period domain. Let $H$ be a supersingular K3 crystal. The associated Tate module $\rT_H\subseteq H$ is a supersingular K3 $\ZZ_p$-lattice in the sense of Ogus (\cf~\cite[3.13]{Ogus78}).   According to \cite[Theorem 3.20]{Ogus78}, the functor \[ H \rightsquigarrow  (T_H, \cK_H) \] where $\cK_H = \ker( T_H \otimes k \to H \otimes k)$ defines an equivalence between the category of supersingular K3 crystals and the category of strictly characteristic subspaces of a supersingular K3 $\ZZ_p$-lattice.
\end{remark}

Using the twisted period map,  we can obtain the following. 
\begin{theorem}\label{prop:SupersingularModuli} 

Let $\sX\to X$ be a $\mu_p$-gerbe over a supersingular abelian surface $X$ over $k$. Then 
\begin{enumerate}
    \item If a primitive vector $v\in \widetilde{\rN}(\sX)$ is positive and isotropic, the coarse moduli space $M_H(\sX,v)$ is an abelian surface.
    
    \item If $\sY\to Y$ is another twisted abelian surface, $\rD^{(1)}(\sX)\simeq \rD^{(1)}(\sY)$ if and only if  there is an isomorphism $$\widetilde{\rH}(\sX,W) \cong \widetilde{\rH}(\sY,W)$$ of K3 crystals.

    \item There is a derived equivalence 
        \[
        \D^{(1)}(\sX_0)\simeq \D^b(X)
        \]
        where $\sX_0\to X_0$ is a $\mu_p$-gerbe over the unique superspecial abelian surface $X_0$.
\end{enumerate}

\end{theorem}

\begin{proof} For (1),  if $\sX\to X$ is an essentially-trivial $\mu_p$-gerbe over a supersingular abelian surface $X$, this can be proved by a standard lifting argument (see also \cite[Proposition 6.9]{FL18}).  When $\sX\to X$ is non-trivial, we can take the universal family of $\mu_p$-gerbes 
\[f\colon \fX\to \AA^1
\]on the connected component $\AA^1 \subset \R^2 u_*\mu_p$ that contains $\sX$ (\cf Corollary \ref{cor:connectedcomponent}). The fibers of $f$ contain $\sX\to X$ and the trivial $\mu_p$-gerbe over $X$.  By taking the relative moduli space of twisted sheaves (with suitable $v$-generic polarization) on $\fX\to \AA^1$, one can obtain the nonemptiness of $M_{H}(\sX,v)$ from the case of essentially trivial gerbes.
\vspace{.1cm}

For the proof of the forward direction of (2), we notice that by Remark \ref{rmk:twistedPeriod}, it is sufficient to find an isomorphism between pairs \[\left(\widetilde{\rN}(\sX), \cK_{\widetilde{\rH}(\sX)}\right) \xrightarrow{\sim} \left(\widetilde{\rN}(\sY), \cK_{\widetilde{\rH}(\sY)}\right),\] which is provided by the de Rham realization of the derived equivalence $\D^{(1)}(\sX) \simeq \D^{(1)}(\sY)$.
The proof of the other direction is identical to the case of K3 surfaces proved in \cite[Theorem 3.5.5]{Bragg20a}. The key is that if $\widetilde{\rH}(\sX,W) \cong \widetilde{\rH}(\sY,W)$, then there exists $v\in \widetilde{\rN}(\sX)$ such that
the induced isomorphism  $$\widetilde{\rH}(\sM_H(\sX,v)^{(-1)}, W) \cong \widetilde{\rH}(\sY, W)$$
of K3 crystals sends $(0,0,1)$ to $(0,0,1)$. The assertion then essentially follows from Ogus' crystalline Torelli theorem for supersingular abelian surfaces (\cf\cite[Theorem 7.3]{Ogus78}), as in \cite[Theorem 3.5.2]{Bragg20a}. We omit the details here. 
\vspace{.1cm}

For (3),  due to (2),  it suffices to find a $\mu_p$-gerbe $\sX_0\to X_0$ such that there is a supersingular K3 crystal isomorphism $$\widetilde{\rH}(\sX_0,W)\cong \widetilde{\rH}(X,W).$$
By Remark \ref{rmk:twistedPeriod}, this is equivalent to find $\sX_0\to X_0$ and an isometry $\widetilde{\rN}(\sX_0)\otimes \ZZ_p\cong \widetilde{\rN}(X)\otimes \ZZ_p$ sending $ \cK_{\widetilde{\rH}(\sX_0)}$ to $\cK_{\widetilde{\rH}(X)}$.

Let us give an explicit construction of $\sX_0\to X_0$ via the twisted period map. If $X$ is superspecial, no further proof is necessary.  Suppose $X$ is not superspecial. Then $\sigma_0(\NS(X)) =2$ by \cite[Proposition 3.7]{Shioda79}. Let $\Lambda$ be the supersingular abelian surface lattice with Artin invariant $2$ and let $\eta \colon \Lambda\xrightarrow{\sim} \NS(X)$ be a $\Lambda$-marking. 
As shown in \cite[Section 2]{RS81} (see also \cite[Proposition 6.1]{FL18}), $\Lambda=U(p)\oplus \Lambda'$ contains $U(p)$ as a direct summand and the image of $U(p)$ in $(p\Lambda^{\vee}/p \Lambda)\otimes k$ is not contained in the strictly characteristic subspace $\rho_{\Lambda}(X,\eta)$. 

Note that the lattice $\Lambda_0=U\oplus \Lambda'$ is a supersingular abelian lattice with Artin invariant $1$. There is a natural isomorphism
\begin{equation}\label{eq:iso}
  \widetilde{\rN}(X)\xrightarrow{\eta\oplus \id} \Lambda\oplus U\cong \Lambda_0\oplus U(p) =\widetilde{\Lambda}_0. 
\end{equation}
and we can identify $\cK_{\widetilde{\rH}(X)}$  with  $\rho_{\Lambda}(X,\eta) $  via the isometry $\eta\oplus \id$.  Let $$\cK\subseteq (p\widetilde{\Lambda}_0^\vee/ p\widetilde{\Lambda}_0)\otimes k$$ be the image of $\cK_{\widetilde{\rH}(X)}$ through the map induced by \eqref{eq:iso}.  By our assumption, $\cK$ does not contain the image of some isotropic vector $e\in U(p)$ and therefore can be viewed as a point in $M^{\langle e\rangle}_{\widetilde{\Lambda}_0}(k)$.  As $\widetilde{\rho}_{\Lambda_0}$ is surjective, there is a $\Lambda_0$-marked supersingular abelian surface $(\sX_0\to X_0, \eta_0)$ such that 
$\widetilde{\rho}_{\Lambda_0}(\sX_0,\eta_0)=\cK $. It is easy to see that $\sX_0\to X_0$ is as desired.
\end{proof}

\section{Shioda's Torelli theorem for abelian surfaces}\label{section:shiodatrick}
In \cite{Shioda78}, Shioda discovered that there is a way to extract information about the $1^{\text{st}}$-cohomology of a complex abelian surface from its $2^{\text{nd}}$-cohomology, called Shioda's trick. This established a global Torelli theorem for complex abelian surfaces via $2^{\text{nd}}$-cohomology, which is also a key step in Pjateckii-\v{S}apiro--\v{S}afarevi\v{c}'s proof of the Torelli theorem for K3 surfaces (\cf\cite[\S 5 Lemma 4, \S 5 Theorem 1]{PS71}). 

The aim of this section is to generalize Shioda's method to all fields and establish an isogeny theorem for abelian surfaces via the $2^{\text{nd}}$-cohomology. We will deal with Shioda's trick for Betti cohomology, \'etale cohomology and crystalline cohomology separately.

\subsection{Recap of Shioda's trick for Hodge isometry}\label{subsec:hodgeshioda}
We first recall Shioda's construction. 
Suppose $X$ is a complex abelian surface. Its singular cohomology ring $\rH^{\bullet}(X,\mathbb{Z})$ is canonically isomorphic to the exterior algebra $\wedge^{\bullet} \rH^1(X,\mathbb{Z})$. Let $\rV$ be a free $\mathbb{Z}$-module of rank $4$. We denote by $\Lambda$ the lattice  $(\wedge^2 \rV, q)$ where  $q \colon \wedge^2 \rV \times \wedge^2 \rV\to \wedge^4 \rV \cong \ZZ$ is the wedge product. 
After choosing a $\ZZ$-basis $\set{v_i}_{1 \leq i \leq 4}$ for $\rH^1(X,\mathbb{Z})$, we have an isometry of $\ZZ$-lattice $\Lambda \xrightarrow{\sim} \rH^2(X,\mathbb{Z})$. The set of vectors 
\[
\set{v_{ij}\coloneqq v_i \wedge v_j}_{0\leq i < j \leq 4}
\] 
clearly forms a basis of $\rH^2(X,\mathbb{Z})$, which will be called an \emph{admissible basis} of $A$ for its second singular cohomology. For another complex abelian surface $Y$, a Hodge isometry
\[
\varphi \colon \rH^2(Y,\mathbb{Z}) \xrightarrow{\sim} \rH^2(X,\mathbb{Z})
\]
will be called \emph{admissible} if $\det(\varphi)=1$, with respect to some admissible bases on $X$ and $Y$.
It is clear that the admissibility of a morphism is independent of the choice of admissible bases. 

In terms of admissible bases, we can view $\varphi$ as an element in $\SO(\Lambda)$. On the other hand, we have the following exact sequence of groups
\begin{equation}
    1 \to \set{\pm 1} \to \SL_4(\ZZ) \xrightarrow{\wedge^2} \SO(\Lambda)
\end{equation}
Shioda observed that the image of $\SL_4(\ZZ)$ in $\SO(\Lambda)$ is a subgroup of index two and does not contain $-\id_{\Lambda}$. From this, he proved the following (\cf\cite[Theorem 1]{Shioda78})
\begin{theorem}[Shioda]\label{thm:shiodaoriginal}
For any admissible integral Hodge isometry $\psi$, there is an isomorphism of integral Hodge structures
\[
\psi \colon \rH^1(Y,\ZZ) \xrightarrow{\sim} \rH^1(X,\ZZ)
\]
such that $\wedge^2(\psi) = \varphi$ or $- \varphi$.
\end{theorem}
This is what we call ``Shioda's trick''. 
As we can assume that a Hodge isometry is admissible after possibly taking the dual abelian variety for one of them (see Example \ref{ex:poincareisomorphism} below), we can obtain the Torelli theorem for complex abelian surfaces by using the weight two Hodge structures, that is, $X$ is isomorphic to $Y$ or its dual $\widehat{Y}$ if and only if there is an integral Hodge isometry $\rH^2(X,\ZZ)\cong \rH^2(Y,\ZZ)$ (\cf\cite[Theorem 1]{Shioda78}).

\subsection{Admissible basis}\label{subsection:admissiblebasis}
To extend Shioda's work to arbitrary fields, we must define admissibility for different cohomology theories (e.g., \'etale and crystalline cohomology).

Let $k$ be a field with $\cha(k)= p \geq 0$. Suppose $X$ is an abelian surface over $k$ and $\ell\nmid p$ is a prime.
For simplicity of notations, we will denote $\rH^{\bullet}(-)_R$ for one of the following  cohomology theories:
\begin{enumerate}
    \item if $k \hookrightarrow \CC$ and $R= \ZZ$ or  any number field $E$, then $\rH^{\bullet}(X)_R = \rH^{\bullet}(X(\CC), R)$ the singular cohomology.
    \item if $R= \Zl$ or $\Ql$, then $\rH^{\bullet}(X)_R = \rH^{\bullet}_{\et}(X_{\bar{k}}, R)$, the $\ell$-adic \'etale cohomology.
    \item if $\cha(k)= p >0$, then we can take $R= W$ a Cohen ring of $k$ or the fraction field $K$ of $W$, then $\rH^{\bullet}(X)_R = \rH^{\bullet}_{\crys}(X/W)$ or $\rH^{\bullet}_{\crys}(X/W) \otimes K$, the crystalline cohomology. 
\end{enumerate}

There is an isomorphism between the cohomology ring $\rH^\bullet(X)_{R}$ and the exterior algebra $\wedge^\bullet\rH^1(X)_R$. We denote by $\tr_{X} \colon \rH^4(X)_R \xrightarrow{\sim} R $ the corresponding trace map. Then the Poincar\'e pairing $\pair{-,-}$ on $\rH^2(X)_R$ can be realized as
\[\pair{\alpha,\beta}=\tr_X(\alpha \wedge \beta).\]
Analogous to \S\ref{subsec:hodgeshioda}, a $R$-basis $\set{v_i}$ of $\rH^1(X)_R$ will be called  a {\it $d$-admissible basis} if it satisfies $$\tr_{X}(v_1 \wedge v_2 \wedge v_3 \wedge v_4) =d$$
for some $d\in R^\ast$. When $d=1$, it will be called an \emph{admissible basis}. For any $d$-admissible (resp.~admissible) basis $\set{v_i}$, the associated $R$-basis  $\set{v_{ij}\coloneqq v_i \wedge v_j}_{i<j}$  of $\rH^2(X)_R$ will also be called $d$-admissible (resp.~admissible). 

\begin{example}
Let $\set{v_1, v_2, v_3,v_4}$ be a $R$-linear basis of $\rH^1(X)_R$. Suppose $$\tr_X(v_1 \wedge v_2 \wedge v_3 \wedge v_4) = t \in R^*.$$ For  any $d \in R^*$,  there is a natural $d$-admissible $R$-linear basis $
\set{\frac{d}{t} v_1, v_2, v_3, v_4 }$
\end{example}

\begin{definition} Let $X$ and $Y$ be abelian surfaces over $k$.
\begin{itemize}
        \item a $R$-linear isomorphism $ \psi\colon \rH^1(X)_R \to \rH^1(Y)_R$ is $d$-admissible if it takes an admissible basis to a $d$-admissible basis. 
        \item a $R$-linear isomorphism $ \varphi \colon \rH^2(X)_R \to \rH^2(Y)_R$ is $d$-admissible if 
        \[
        \tr_Y \circ \wedge^2(\varphi) = d \tr_X
        \]
        for some $d \in R^*$, or equivalently, it sends an admissible basis to a $d$-admissible basis. When $d=1$, it will also be called admissible.
\end{itemize}
The set of $d$-admissible isomorphisms are denoted by   $\Iso^{\ad,(d)}(\rH^i(X)_R, \rH^i(Y)_R)$ 
accordingly.

\end{definition}
For any isomorphism $\varphi \colon \rH^2(X)_R \xrightarrow{\sim} \rH^2(Y)_R$, let $\det(\varphi)$ be the determinant of the matrix with respect to some admissible bases.
It is not hard to see $\det(\varphi)$ is independent of the choice of admissible bases, and $\varphi$ is admissible if and only if $\det(\varphi)=1$.

\begin{example}\label{ex:poincareisomorphism} Let $\set{v_i}$ be an admissible basis of $\rH^1(X)_R$.
For the dual abelian surface $\widehat{X}$, the dual basis $\set{v_i^*}$ with respect to the Poincar\'e pairing naturally forms an admissible basis of $\widehat{X}$, under the identification $\rH^1(X)_{R}^\vee \cong \rH^1(\widehat{X})_R$.
Let \[
\varphi^{\cP} \colon \rH^2(X)_R \to \rH^2(\widehat{X})_R
\]
be the isomorphism induced by the Poincar\'e bundle $\cP$ on $X \times \widehat{X}$. A direct computation (see e.g.~\cite[Lemma 9.3]{FMtransform}) shows that $\varphi^{\cP} $ is nothing but
\[
-\D \colon \rH^2(X)_R \xrightarrow{\sim} \rH^2(X)_R^\vee \cong \rH^2(\widehat{X})_R,
\]
where $\D$ is the Poincar\'e duality. For an admissible basis $\set{v_i}$ of $X$, its $R$-linear dual $\set{v^*_i}$ with respect to Poincar\'e pairing forms an admissible basis of $\widehat{X}$. By our construction, we can see
\[
\D(v_{12},v_{13}, v_{14}, v_{23},v_{24},v_{34})= (v_{34}^*, -v_{24}^*,v_{23}^*,v_{14}^*, -v_{13}^*, v_{12}^*),
\]
which implies that $\D$ is of determinant $-1$ under these admissible bases. Thus the determinant of $\varphi^{\cP}$ is not admissible.
\end{example}

\begin{example}
Let $f \colon X\rightarrow Y$ be an isogeny of degree $d$ for some $d \in \ZZ_{\geq 0}$ between two abelian surfaces. If $d$ is coprime to $\ell$, then it will induce an isomorphism
\[
f^\ast \colon \rH^2(Y)_{\Zl} \xrightarrow{\sim} \rH^2(X)_{\Zl},
\]
that is $d$-admissible.
If, in addition, $d = n^2$, then $\frac{1}{n}f^\ast$ will be an admissible $\ZZ_\ell$-integral isometry with respect to the Poincar\'e pairing. Moreover, if $d = k^4$, then $\frac{f}{k}$ will be a $\ZZ_{(\ell)}$-isogeny such that its pull-back is admissible integral.
\end{example}

\begin{example}
Suppose $X$ is an abelian surface over a perfect field $k$ with $\cha(k)=p >0$. Then $F$-crystal $\rH^1(X)_W$ together with the trace map $$\tr_X \colon \rH^4(X)_W \xrightarrow{\sim}W$$ form an abelian crystal (of genus $2$) in the sense of \cite[\S 6]{Ogus78}. We can see that an isomorphism of $F$-crystals $\rH^1(X)_W \xrightarrow{\sim} \rH^1(Y)_W$ is admissible if and only if it is an isomorphism between abelian crystals, i.e., it is compatible with trace maps.
\end{example}

\subsection{More on admissible basis of $F$-crystals}\label{ex:adm-crys} 
In contrast to $\ell$-adic \'etale cohomology, the semilinear structure on crystalline cohomology from its Frobenius is more tricky to work with. Therefore, it seems necessary for us to spend more words on the interaction of Frobenius with admissible bases.

Suppose $k$ is a perfect field with $\cha(k) =p >0$, we have the following Frobenius pull-back diagram:
\[
\begin{tikzcd}
X \ar[dr,"F_X^{(1)}"] \ar[drr,bend left,"F_X"] \ar[ddr,bend right] & &\\
&X^{(1)} \ar[r] \ar[d] & X \ar[d] \\
&\spec(k) \ar[r,"\sigma"] & \spec(k)
\end{tikzcd}
\]
Via the natural identification $\rH^1_{\crys}(X^{(1)}/W) \cong \rH^1_{\crys}(X/W) \otimes_{\sigma} W$, the $\sigma$-linearization of Frobenius action on $\rH^1_{\crys}(X/W)$ can be viewed as the injective $W$-linear map
\[
F^{(1)}\coloneqq \left(F_X^{(1)} \right)^* \colon \rH^1_{\crys}(X^{(1)}/W) \hookrightarrow \rH^1_{\crys}(X/W).
\]
If $k$ is not perfect, then after passing to $W(\bar{k})$ or equivalently choosing a Frobenius lift on the Cohen ring $W$, we also get a Frobenius action on $\rH^1_{\crys}(X/W)$, whose linearization is given by the relative Frobenius morphism.

There is a  decomposition $\rH^1_{\crys}(X/W)=H_0(X) \oplus H_1(X)$  such that
\begin{equation}\label{eq:crysdecomposition}
F^{(1)}\left(\rH^1_{\crys}(X^{(1)}/W)\right) \cong H_0(X) \oplus pH_1(X),
\end{equation}
and $\rank_{W}H_i =2$ for $i=0,1$, which is related to the Hodge decomposition of the de Rham cohomology of $X/k$ by Mazur's theorem; see \cite[\S 8, Theorem 8.26]{BerthelotOgus}. 

The Frobenius map can be expressed in terms of admissible basis. We can choose an admissible basis $\set{v_i}$ of $\rH^1_{\crys}(X/W)$ such that
\[
v_1, v_2 \in H_0(X) \quad \text{ and } \quad v_3, v_4 \in H_1(X).
\]
Then $\set{p^{\alpha_i} v_i} \coloneqq \set{ v_1 , v_2 , p v_3 , pv_4}$ forms an admissible basis of $\rH^1_{\crys}(X^{(1)}/W)$ under the identification \eqref{eq:crysdecomposition}, since $\tr_{X^{(1)}} \circ \wedge^4F^{(1)} = p^2\sigma_W\circ\tr_X$. In term of these basis, the Frobenius map can be written as
\begin{equation}\label{eq:FrobeniusAction}
F^{(1)}(p^{\alpha_i} v_i) = \sum_{j} c_{ij} p^{\alpha_j} v_j,
\end{equation}
where $C_X=(c_{ij})$ forms an invertible $4 \times 4$-matrix with coefficients in $W$.

Suppose $Y$ is another abelian surface over $k$, $\psi \colon \rH^1_{\crys}(X/W) \to \rH^1_{\crys}(Y/W)$ is an admissible map, and $\psi^{(1)}$ is the induced map $\psi\otimes_{\sigma}W \colon \rH^1_{\crys}(X^{(1)}/W) \to \rH^1_{\crys}(Y^{(1)}/W)$. Denote by $M$ and $M'$  the matrix of $\psi$ and $\psi^{(1)}$ with respect to the chosen admissible bases, respectively.
\begin{lemma}\label{lemma:Frobeniuslinearalg}
The map $\psi$ commutes with Frobenius if and only if $C_Y M' C_X^{-1} = M$.
\end{lemma}

\begin{proof}
By definition,  $\psi$ commutes with Frobenius if and only if $(F_{Y}^{(1)})^*\circ \psi^{(1)} = \psi \circ (F_X^{(1)})^*$. The statement is then clear from  \eqref{eq:FrobeniusAction} . 
\end{proof}

\subsection{Generalized Shioda's trick} 
Let us review some basic properties of the special orthogonal group scheme over an integral domain. Our main reference is \cite[Appendix C]{luminysga3smf}. 

Let $\Lambda $ be an even $\mathbb{Z}$-lattice of rank $2n$. Then we can associate it with a vector bundle $\underline{\Lambda}$ on $\spec(\bbZ)$ with constant rank $2n$ equipped with a quadratic form $q$ over $\spec(\bbZ)$ obtained from $\Lambda$. Then the functor
\[
A \mapsto \set*{ g \in \GL(\Lambda_A) \big| q_A(g\cdot x) =  q_A(x) \text{ for all } x \in \Lambda_A}
\]
is representable by a $\mathbb{Z}$-subscheme of $\GL(\Lambda)$, denoted by $\Og(\Lambda)$. There is a homomorphism between the $\bbZ$-group schemes
\[
D_{\Lambda} \colon \Og(\Lambda) \to \underline{\bbZ/2\bbZ},
\]
which is called the Dickson morphism (see p313 in loc.~cit. for the definition). Roughly speaking, 
\[
D_{\Lambda}(g) = \begin{cases}
    0 & \text{ if $g$ is a product of an even number of reflections}\\
    1 & \text{ if $g$ is a product of an odd number of reflections}
\end{cases}
\]
for a point $g \in \Og(\Lambda)$ over a field in characteristic zero.
The Dickson morphism is surjective as $\Lambda$ is even and its construction is compatible with any base change (see Proposition C.2.8 in loc.~cit.). The \emph{special orthogonal group scheme} over $\bbZ$ with respect to $\Lambda$ is defined to be the kernel of $D_{\Lambda}$, which is denoted by $\SO(\Lambda)$. Moreover, we have
\[
\SO(\Lambda)_{\bbZ[\frac{1}{2}]} \cong \ker\left(\det \colon \Og(\Lambda) \to \mathbb{G}_m\right)_{\bbZ[\frac{1}{2}]}.
\]
It is well-known that $\SO(\Lambda) \to \spec(\bbZ)$ is smooth in relative dimension $\frac{n(n-1)}{2}$ and with connected fibers; see Theorem C.2.11 in loc. cit. for example.

For any $\ell$, the special orthogonal group scheme $$\SO(\Lambda_{\Zl}) \cong \SO(\Lambda)_{\Zl}$$ is smooth over $\Zl$ with connected fibers, which implies that its generic fiber $\SO(\Lambda_{\mathbb{Q}_{\ell}})$ is connected. Thus, $\SO(\Lambda_{\Zl})$ is clearly connected as a group scheme over $\Zl$ as $\SO(\Lambda_{\mathbb{Q}_{\ell}}) \subset \SO(\Lambda_{\Zl})$ is dense. 

The special orthogonal group scheme admits a universal covering (i.e., a simply connected central isogeny)
\[
\Spin(\Lambda) \to \SO(\Lambda).
\]
See Appendix C.4 in loc. cit.  for construction.

\begin{lemma}\label{lemma:isogenywedge} \label{lemma:ellimageofSL}
 Let $\rV$ be free $\ZZ$-module of rank $4$ and $\Lambda = \wedge^2 \rV$. Let $R$ be a ring of coefficients as listed in \S\ref{subsection:admissiblebasis}.  There is an exact sequence of smooth $R$-group schemes
\[
1 \to \mu_{2,R} \to \SL(\rV)_{R} \xrightarrow{\wedge^2(-)_{R}} \SO(\Lambda)_R \to 1.
\]
(as fppf-sheaves if $\frac{1}{2} \notin R$.)
Moreover, there is an exact sequence
\begin{equation}\label{eq:sequencespinor}
1 \to \set{ \pm \id_4} \to \SL(\rV)(R) \xrightarrow{\wedge^2(-)_{R}} \SO(\Lambda)(R) \to R^*/(R^*)^2.
\end{equation}
\end{lemma}
\begin{proof}

For the first statement, it suffices to assume $R= \spec(\bar{k})$ for an algebraically closed field $\bar{k}$, where it is clear from a computation. Note that we have an exact sequence on rational points (\cf\cite[Proposition 3.2.2]{Giraud})
\[
1 \to \mu_2(R) \to \SL(\rV)(R) \to \SO(\Lambda)(R) \to \rH^1(\spec(R), \mu_2).
\]
Notice that for the rings of coefficients listed in \S\ref{subsection:admissiblebasis}, we have $\Pic(R)[2]=0$. Therefore,
\[
\rH^1_{\fl}(\spec(R),\mu_2) \cong R^*/(R^*)^2
\]
from the Kummer sequence for $\mu_2$.

For the last statement, it is sufficient to see that there is an isomorphism of $R$-group schemes $\SL(\rV)_R \xrightarrow{\sim} \Spin(\Lambda)_R$ such that the following diagram commutes
\[
\begin{tikzcd}
\SL(V)(R) \ar[rr,"\sim"] \ar[rd] & & \Spin(\Lambda)(R) \ar[ld] \\
&\SO(\Lambda)(R) &  
\end{tikzcd}
\]
The group scheme $\SL(\rV)$ is simply-connected (as its geometric fibers are semisimple algebraic group of type $A_{3}$). Thus, the central isogeny $\SL(\rV)_R \to \SO(\Lambda)_R$ forms the universal covering of $\SO(\Lambda)_R$, which induces an isomorphism $\SL(\rV)_R \xrightarrow{\sim} \Spin(\Lambda)_R$ by using the Isomorphism Theorem over a general ring (see, e.g.,\cite[Theorem 6.1.16, 6.1.17]{luminysga3smf}).
\end{proof}

\begin{remark}\label{rmk:trickwitt} 
When $R = \Zl$, we have
\[
\Zl^*/(\Zl^*)^2 \cong \begin{cases} \set{\pm 1} & \text{ if } \ell \neq 2, \\
\set{\pm 1} \times \set{\pm 5} & \text{ if } \ell =2.
\end{cases}
\]
Thus the image of $\SL(\rV)(\Zl)$ is a finite index subgroup in $\SO(\Lambda)(\Zl)$.
\end{remark}
\begin{remark}\label{rmk:wittvector}
When $R =W(k)$, we have
\[
W(k)^*/(W(k)^*)^2 \cong \begin{cases} \set{1,\epsilon} & \text{ if $k=\mathbb{F}_{p^s}$ for $p >2, s\geq 1$} \\
\set{1} & \text{ if $k= \bar{k}$ or $k^{s} = k, \cha(k)>2$.}
\end{cases}
\]
where $\epsilon \in \ZZ$ such that $\epsilon \not\equiv y^2\!\mod p^s$ for an integer $y$, as $W(k)$ is Henselian. 
Thus, the wedge map $\SL(\rV)(W) \to \SO(\Lambda)(W)$ is surjective when $k = \bar{k}$.
\end{remark}

Let $X$ and $Y$ be abelian surfaces over $k$.
Let $V_R = \rH^1(X)_R$. We can see the set 
\[
\Iso^{\ad,(d)}(\rH^1(X)_R, \rH^1(Y)_R)
\]
is a (right) $\SL(V_R)$-torsor if it is nonempty.
The wedge product provides a natural map
\[
\wedge^2 \colon \Iso^{\ad,(d)}\left(\rH^1(X)_R, \rH^1(Y)_R \right) \to \Iso^{\ad, (d)}\left(\rH^2(X)_R, \rH^2(Y)_R\right).
\]

Let $\set{v_i}$ be an admissible basis of $\rH^1(X)_R$ and let $\set{v_i'}$ be a $d$-admissible basis of $\rH^1(Y)_R$, respectively. There is an $d$-admissible isomorphism $\psi_0 \in  \Iso^{\ad,(d)}(\rH^1(X)_R, \rH^1(Y)_R)$ such that $\psi_0(v_i)= v_i'$.
For a $d$-admissible isometry $\varphi \colon \rH^2(X,R) \to \rH^2(Y,R)$, we can see 
\[
 \varphi =  \wedge^2(\psi_0) \circ g,~\text{for some $g \in \SO(\Lambda_R)$}.
\]
In this way, any  $d$-admissible isomorphism $\varphi$ can be identified with the (unique) element $g \in \SO(\Lambda)(R)$ when the admissible bases are fixed.
This allows us to deal with $d$-admissible isomorphisms group-theoretically. In particular, we have the following notion of the spinor norm.
\begin{definition}
 The \emph{spinor norm} of the $d$-admissible isomorphism $\varphi$ is defined to the image of $g$ under $\SN \colon \SO(\Lambda)(R) \to R^*/(R^*)^2$, denoted by $\SN(\varphi)$.
\end{definition}

\begin{lemma}
The spinor norm $\SN(\varphi)$ is independent of the choice of admissible bases.
\end{lemma}
\begin{proof}

For different choice of admissible bases, we can see the resulted $\widetilde{g} = K g K^{-1}$ for some $K \in \SO(\Lambda_R)$. Therefore, $\SN(\widetilde{g}) = \SN(g)$.
\end{proof}
\begin{remark}\label{rmk:CDdecompositionSpinornorm}
When $R$ is a field, the spinor norm can be computed by the Cartan--Dieudonn\'e decomposition. That means we can write any $ g \in \SO(\Lambda)(R)$ as a composition of reflections:
\[
\bR_{b_n} \circ \bR_{b_{n-1}} \circ \cdots \circ \bR_{b_1}
\]
for some non-isotropic vectors $b_1, \cdots, b_{n} \in \Lambda_R$, and $\SN(g) = \left[(b_1)^2\cdots (b_{n-1})^2 (b_{n})^2 \right]$. 
\end{remark}

\begin{lemma}\label{lemma:spinornormone}
The $d$-admissible isomorphism $\varphi$ is a wedge of some $d$-admissible isomorphism $\psi \colon \rH^1(X,R) \to \rH^1(Y,R)$ if and only if $\SN(\varphi)=1$. 
\end{lemma}

\begin{proof}
The exact sequence \eqref{eq:sequencespinor} shows that if $\SN(\varphi) = \SN(g) =1$, then there is some $h \in \SL(V_R)$ such that $\wedge^2(h)=g$. Thus, we can take $\psi= \psi_0 \circ h$ when $\SN(\varphi)=1$, and see that \[
\wedge^2(\psi) =  \wedge^2(\psi_0) \circ \wedge^2(h)= \varphi. 
\]
The converse is clear. 
\end{proof}

\subsubsection{Isogenies category}\label{subsec:isogeny}
Recall that the isogeny category of abelian varieties $\abelian_{\mathbb{Q},k}$  consists of all abelian varieties over a field $k$ as objects, and the homomorphism sets are
\[
\Hom_{\abelian_{\mathbb{Q},k}}(X,Y) \coloneqq \Hom_{\abelian_{k}}(X,Y) \otimes_{\mathbb{Z}} \mathbb{Q},
\]
where $\Hom_{\abelian_{k}}(X,Y)$ is the abelian group of homomorphisms from $X$ to $Y$ with the natural addition. We may also write $\Hom^0(X,Y)$ for $\Hom_{\abelian_{\mathbb{Q},k}}(X,Y)$ if there is no confusion in the definition field $k$. 
 \begin{definition}
     Let $R$ be a commutative ring with units. A \emph{$R$-isogeny} from $X$ to $Y$ is an invertible element $f \in \Hom_{\abelian_k}(X,Y) \otimes R$ i.e., there is an $g \in \Hom_{\abelian_k}(Y,X)\otimes R$ such that $f \circ g = \id_{Y}$ and $g \circ f = \id_{X}$. 
     
    A $\QQ$-isogeny is called a quasi-isogeny, while $\ZZ_{(\ell)}$-isogeny is called {\it a prime-to-$\ell$ quasi-isogeny}.   For any (prime-to-$\ell$) quasi-isogeny $f$, we can find a minimal integer $n$ (resp.~$\ell\nmid n$) such that
 \[
 n f \colon X \to Y
 \]
 is an isogeny (resp. of degree prime-to-$\ell$). 
 \end{definition}

 When $k= \CC$, with the uniformization of complex abelian varieties, we have a canonical bijection
 \[
 \Hom_{\abelian_{\QQ,\CC}}(X,Y) \xrightarrow{\sim} \Hom_{\hdg}\left(\rH^1(Y,\QQ), \rH^1(X,\QQ) \right),
 \]
 where the right-hand side is the set of $\QQ$-linear morphisms that preserve Hodge structures. Then the integer $n$ for $f$ is also the minimal integer such that $(nf)^*(\rH^1(Y,\ZZ)) \subseteq \rH^1(X,\ZZ)$.

\subsection{Shioda's trick for Hodge isogenies} 
Suppose $k = \CC$. Let $d$ be an integer. A {\it Hodge isogeny of degree d} is an isomorphism of $\QQ$-Hodge structures
\[
\varphi\colon\rH^2(X,\QQ)\xrightarrow{\sim}\rH^2(Y,\QQ)
\]
such that 
\[
\langle x, y \rangle = d \langle \varphi(x), \varphi(y) \rangle.
\]
In particular, if $d=1$, then it is the classical Hodge isometry that we usually talk about. Clearly, a $d$-admissible rational Hodge isomorphism is a Hodge isogeny of degree $d$. 
In terms of spinor norms, we can generalize Shioda's theorem \ref{thm:shiodaoriginal} to admissible rational Hodge isogenies.
\begin{proposition}[Shioda's trick on admissible Hodge isogenies]
\label{prop:hodgeisogeny}
\leavevmode
\begin{enumerate}
    \item A $d$-admissible Hodge isogeny of degree $d$ 
    \[
    \varphi \colon \rH^2(X, \QQ) \xrightarrow{\sim} \rH^2(Y,\QQ)
    \]
    is a wedge of some rational Hodge isomorphism  $\psi \colon \rH^1(X,\QQ) \xrightarrow{\sim} \rH^1(Y,\QQ)$, if its spinor norm is trivial. In this case, the Hodge isogeny is induced by a quasi-isogeny of degree $d^2$. 

    \item When $d=1$, any admissible Hodge isometry $\varphi:\rH^2(X,\QQ) \xrightarrow{\sim} \rH^2(Y,\QQ)$ is induced by an isogeny  $ f\colon Y \to X$ of degree $n^2$ for some integer $n$ such that   $\varphi=\frac{f^\ast}{n}$.  
\end{enumerate}
\end{proposition}

\begin{proof}
Under the assumption of $(1)$, we can find a $d$-admissible isomorphism $\psi$ by applying the Lemma \ref{lemma:spinornormone}. It remains to prove that $\psi$ preserves the Hodge structure, which is essentially the same as in \cite[Theorem 1]{Shioda78}.

For $(2)$, we suppose the spinor norm $\SN(\varphi)=n \QQ^{*2} \in \QQ^*/\QQ^{*2}$. Let $E= \QQ(\sqrt{n})$. We can see that the base change $\rH^2(X,E) \xrightarrow{\sim} \rH^2(Y,E)$ is a Hodge isometry with coefficients in $E$ such that $\SN(\varphi) =1 \in E^*/(E^*)^2$. Then by applying Lemma \ref{lemma:spinornormone}, we will obtain an admissible (fixing the admissible bases for $\rH^1(X,\QQ)$ and $\rH^1(Y,\QQ)$) Hodge isomorphism $\psi \colon \rH^1(X,E) \xrightarrow{\sim} \rH^1(Y,E)$. Let
\[
\sigma\colon a+ b\sqrt{n} \rightsquigarrow a -b \sqrt{n}
\]be the generator of $\Gal(E/\QQ)$. As we have fixed the $\QQ$-linear admissible bases, the wedge map
\[
\SL_4(E) \xrightarrow{\wedge^2} \SO(\Lambda)(E)
\]
is defined over $\QQ$, and so is $\sigma$-equivariant. Let $g$ be the element in $\SL_4(E)$ that corresponds to $\psi$. As $\wedge^2(g) \in \SO(\Lambda) \subset \SO(\Lambda_E)$, we can see
\[
(\wedge^2(\sigma(g)) =\sigma(\wedge^2(g)) = \wedge^2(g).
\]
which implies that $\sigma(g) g^{-1} = \pm \id_4$ since $\ker(\wedge^2) =\set{\pm \id_4} $. If $\sigma(g) = g$, then $g \in \SL_4(\QQ)$ and the statement is trivially valid. If $\sigma(g) = -g$, then $g_0=\sqrt{n} g$ is lying in $\GL_4(\QQ)$.  Let
\[
\psi_0 \colon \rH^1(X,\QQ) \to \rH^1(Y,\QQ)
\]
be the corresponding element of $g_0$ in $\Iso^{\ad,(n^2)}\left(\rH^1(X,\QQ), \rH^1(Y,\QQ)\right)$. As $\wedge^2 \psi_0=n \varphi$ is a Hodge isogeny, part (1) then implies that $\psi_0$ is also a Hodge isomorphism. Thus, $\psi_0$ increases to a quasi-isogeny $f_0 \colon Y \to X$ and we have 
\[
\varphi = \wedge^2(\psi)= \frac{f_0^*}{n} \colon \rH^2(X,\QQ) \to \rH^2(Y,\QQ).
\]
Replacing $f_0$ by multiplication $m f_0$ for some integer $m$, we can get an isogeny of degree $(m^2n)^2$.
\end{proof}
\begin{remark}
If a Hodge isometry $\psi \colon \rH^2(X,\QQ) \xrightarrow{\sim} \rH^2(Y,\QQ)$ is not admissible, that is, its determinant is $-1$ with respect to some admissible bases, then we can take its composition with the isometry $\psi^{\cP}$ induced by the Poincar\'e bundle as in Example \ref{ex:poincareisomorphism}. After that, we can see that $\psi^{\cP} \circ \psi$ is admissible and is induced by an isogeny $f \colon \widehat{Y} \to X$.
\end{remark}
\subsection{$\ell$-adic and $p$-adic Shioda's trick}
For the integral $\ell$-adic \'etale cohomology, we have the following statement similar to Shioda's trick for integral Betti cohomology.

\begin{proposition}[$\ell$-adic  Shioda's trick]
\label{prop:etseconcdtofirst}
Suppose $\ell \neq 2$. For any $d$-admissible $\Zl$-linear isomorphism
\[
\varphi_{\ell} \colon \rH^2_{\et}(Y_{k^s},\Zl) \xrightarrow{\sim} \rH^2_{\et}(X_{k^s},\Zl),
\]
there are an integer $u$ and a $(u^2d)$-admissible $\Zl$-isomorphism $\psi_{\ell}$, such that $\wedge^2(\psi_{\ell}) = u \varphi_{\ell}$. Moreover, if  $\varphi_{\ell}$ is $\Gal(k^s/k)$-equivariant, then $\psi_{\ell}$ is also $\Gal(k^s/k)$-equivariant after replacing $k$ with some finite extension.
\end{proposition} 
\begin{proof}
One can choose an element $u \in (\ZZ\setminus \set{0}) \cap \ZZ_{\ell}^* $ that is not a square in $\ZZ_{\ell}$, e.g., those satisfying equation $u^{\frac{\ell -1}{2}} \equiv -1 \mod \ell$ as $\ell \neq 2$. As $\Zl^*/(\Zl^*)^2 \cong \set{\pm 1}$ for any $\ell \neq 2$, $\varphi_{\ell}$ or $u \varphi_{\ell}$ is of spinor norm one. Then the first statement follows from Lemma \ref{lemma:spinornormone}.

Suppose $\varphi_{\ell}$ is $\Gal(k^s/k)$-equivariant. We may assume $\wedge^2(\psi_{\ell}) = \varphi_{\ell}$ for simplicity. For any $g \in \Gal(k^s/k)$, we have \[\wedge^2(g^{-1} \psi_{\ell} g) = g^{-1} \wedge^2(\psi_{\ell}) g =\wedge^2(\psi_\ell).\]Therefore,  $g^{-1}\psi_{\ell} g = \pm \psi_{\ell}$.  By passing to a finite extension $k'/k$, we always have $g^{-1}\psi_{\ell} g =  \psi_{\ell}$ for all $g\in \Gal(k^s/k')$ which proves the assertion.  
\end{proof}

For $F$-crystals attached to abelian surfaces, we can also use Shioda's trick.

\begin{proposition}[$p$-adic Shioda's trick]
\label{prop:cryssecondtofirst}
Let $k$ be a finite field or an algebraically closed field, such that $\cha(k) =p >2$. For any $d$-admissible $W$-linear isomorphism \[\varphi_p \colon \rH^2_{\crys}(Y/W) \xrightarrow{\sim} \rH^2_{\crys}(X/W),\] there exist an integer $u$ and a $W$-linear isomorphism $\psi_p \colon \rH^1_{\crys}(Y/W)\xrightarrow{\sim} \rH^1_{\crys}(X/W)$ that is $(u^2d)$-admissible, satisfying $\wedge^2(\psi_p) = u\varphi_p$. Furthermore, if $k$ is algebraically closed, then $u =1$.

Moreover, if $\varphi_p$ is compatible with Frobenius and $\FF_{p^2} \subseteq k$, then there is $\xi \in \ZZ_{p^2}^* \subseteq W(k)$ such that $\xi \psi_p$ is compatible with Frobenius and $\xi^2 \in \ZZ_p^*$.
\end{proposition}

\begin{proof} 
The first statement follows from a similar reason as in Proposition \ref{prop:etseconcdtofirst} as $W^*/(W^*)^2 \subseteq \set{1, \epsilon}$ (see Remark \ref{rmk:wittvector}). 

For the second statement, we assume $\wedge^2(\psi_p) = \varphi_p$. If $\varphi_p$ commutes with the Frobenius action, then we have
\[
\wedge^2( C_X^{-1} \cdot \psi_p^{(1)} \cdot C_Y) = \varphi_p.
\]
as in \S \ref{ex:adm-crys}. 
Thus $C_X^{-1}\cdot \psi_p^{(1)}\cdot C_Y = \pm \psi^{(1)}_p$, which implies
\[
\psi_p \circ F_X^{(1)} = \pm F_Y^{(1)} \circ \psi_p^{(1)}
\]by Lemma \ref{lemma:Frobeniuslinearalg}.

If $F_X^{(1)} \circ \psi_p^{(1)} = \psi_p \circ F_Y^{(1)}$, then we need to do nothing. If $F_X^{(1)} \circ \psi_p^{(1)} = -\psi_p \circ F_Y^{(1)}$, then we can take $\xi \in \ZZ_{p^2}^* \subseteq W(k)$ such that $\xi^{p-1} = -1$. This implies
\[
F_X^{(1)} \circ (\xi \psi_p)^{(1)} = \xi^p F_X^{(1)} \circ \psi = (\xi \psi_p) \circ F_Y^{(1)}. 
\] 
Note that $\xi^2 \in \ZZ_{p}^*$ as $\sigma(\xi^2) = \xi^2$ and $\xi^{2p+2} =1$. Therefore, we can conclude.
\end{proof}

Combined with Tate's isogeny theorem, we have the following direct consequences of Propositions \ref{prop:etseconcdtofirst} and \ref{prop:cryssecondtofirst}. It includes a special case of Tate's conjecture.
\begin{corollary}\label{cor:shiodatrickforisog}
Suppose $k$ is a finitely generated field over $\mathbb{F}_p$ with $p>2$. Let $\ell \neq 2$ be a prime not equal to $p$.
\begin{enumerate}
    \item  For any admissible isometry of $\Gal(k^s/k)$-modules 
\[
\varphi_{\ell} \colon \rH^2_{\et}(Y_{k^s},\Zl) \xrightarrow{\sim} \rH^2_{\et}(X_{k^s},\Zl),
\]
we can find a $\Zl$-isogeny $f_{\ell} \in \Hom_{k'}(X_{k'}, Y_{k'})\otimes \Zl$
for some finite extension $k'/k$, which induces $u\varphi_{\ell}$ for some integer $u$ prime-to-$\ell$. In particular, $\varphi_{\ell}$ is algebraic.
    \item If $k$ is finite, then for any admissible isometry
    \[
\varphi_{p} \colon \rH^2_{\crys}(Y/W) \xrightarrow{\sim} \rH^2_{\crys}(X/W),
\]
which is compatible with Frobenius, we can find a $\ZZ_{p}$-isogeny $f_p \in \Hom_{k'}(X_{k'}, Y_{k'}) \otimes \ZZ_{p}$ over some finite extension $k'/k$, such that
\[
\epsilon f_p^*|_{\rH^2_{\crys}(Y/W)} = u\varphi_p
\]
for some prime-to-$p$ integer $u$ and $\epsilon \in \ZZ_p^*$. In particular, $\varphi_p$ is algebraic.
\end{enumerate}
\end{corollary}

\begin{proof}
For $(1)$, Proposition \ref{prop:etseconcdtofirst} implies that there is an isomorphism \[
\psi_{\ell} \colon \rH^1_{\et}(Y_{k^s}, \Zl) \xrightarrow{\sim} \rH^1_{\et}(X_{k^s}, \Zl),
\]
that induces $u \varphi_{\ell}$, which is $\Gal(k^s/k)$-equivariant after a finite extension of $k$. Then $f_{\ell}$ exists by the following canonical bijection (\cf\cite{Zarhin76} and \cite[VI, \S 3 Theorem 1]{rationalpoints})
\[
\Hom^0(X,Y) \otimes \Zl \xrightarrow{\sim} \Hom_{\Gal(k^s/k)}\left( \rH^1_{\et}(Y_{k^s},\Zl), \rH^1_{\et}(X_{k^s},\Zl)\right)
\label{eq:isogenyetale}.
\]

For $(2)$, we may assume that $\ZZ_{p^2} \subseteq W(k)$ after taking a finite extension of $k$. The Proposition \ref{prop:cryssecondtofirst} implies that there is an isomorphism \[ \psi_p \colon \rH^1_{\crys}(Y/W)\xrightarrow{\sim} \rH^1_{\crys}(X/W) \] that induces $ u \varphi_p$, and $\xi \in \ZZ_{p^2}^*$ such that $\xi \psi_p$ is compatible with Frobenius.

Since $k$ a finite field, there are canonical isomorphisms
\begin{equation}\label{eq:deJongthm}
  \Hom^0(X,Y) \otimes \Zp \xrightarrow{\sim} \Hom_k\left(X[p^{\infty}], Y[p^{\infty}]\right) \xrightarrow{\sim} \Hom_{F,V}\left(\rH^1_{\crys}(Y/W), \rH^1_{\crys}(X/W)\right).
\end{equation}
Here the first isomorphism is from $p$-adic Tate's isogeny theorem (\cf\cite[Theorem 2.6]{deJong}) and the second from the faithfulness of Dieudonn\'e functor over $W$ (\cf\cite[Theorem]{deJong2}).
The canonical bijection \eqref{eq:deJongthm} implies that $\xi \psi_p$ is induced by a $\Zp$-isogeny $f_p \in \Hom^0(X,Y) \otimes \Zp$. Therefore
\[
f_p^*|_{\rH^2_{\crys}(Y/W)} = \xi^2 u \varphi_p.
\]
The $\ZZ_p$-isogeny $f_p$ is what we require.
\end{proof}

\begin{remark}
In \cite{Zarhin17}, Zarhin introduces the notion of \emph{almost isomorphism}. Two abelian varieties over $k$ are called almost isomorphic if their Tate modules $T_{\ell}$ are isomorphic as Galois modules (replaced by $p$-divisible groups when $\ell=p$). The proposition \ref{prop:etseconcdtofirst} and \ref{prop:cryssecondtofirst} imply that it is possible to characterize almost isomorphic abelian surfaces by their $2^{\text{nd}}$-cohomology groups.
\end{remark}

\section{Derived isogeny in characteristic zero}\label{section:derivedchar0}

In this section, we follow \cite{LV19} and \cite{Hu19} to prove the twisted Torelli theorem for abelian surfaces over algebraically closed fields of characteristic zero.

\subsection{Over $\CC$: Hodge isogeny versus derived isogeny}\label{subsec:hodgederived}
Let $X$ and $Y$ be complex abelian surfaces.
 Throughout this section, let $\Lambda = U^{\oplus 3}$ be the direct sum of three hyperbolic lattices. 
\begin{definition}\label{def:reflectiveHodge}
 A rational Hodge isometry $\varphi \colon \rH^2(X,\mathbb{Q}) \to \rH^2(Y,\mathbb{Q})$ is called \emph{reflective} if it is a reflection on $\Lambda$ along a non-isotropic vector $b\in \Lambda$: 
\[
 \bR_b \colon \Lambda_{\mathbb{Q}} \xrightarrow{\sim} \Lambda_{\mathbb{Q}} \quad x  \mapsto x - \frac{2(x,b)}{(b,b)} b,
\]
after choosing the markings $\rH^2(X,\ZZ)\cong \Lambda$ and $\rH^2(Y,\ZZ)\cong \Lambda$.
\end{definition}

A key lemma is 
\begin{lemma}\label{lemma:reflectiveHodgeisometry}
Any reflective Hodge isometry $\varphi$ induces a Hodge isometry on twisted Mukai lattices
\[
\widetilde{\varphi} \colon \widetilde{\rH}(X,\mathbb{Z};B) \to \widetilde{\rH}(Y,\mathbb{Z};B'),
\]
for some $B\in \rH^2(X,\QQ)$ and $B'= -\varphi(B)$ such that the restriction of $\widetilde{\varphi}_{\QQ} \colon \widetilde{\rH}(X,\QQ) \xrightarrow{\sim}\widetilde{\rH}(Y,\QQ)$ on $\rH^2(X,\QQ)$ is equal to $\varphi$.
\end{lemma}
\begin{proof}
This is due to the work in \cite[\S 1.2]{Hu19}. Since this is a purely linear-algebraic argument for twisted Mukai lattices, it works for abelian surfaces without changes.
 Let us briefly recall the construction of $\widetilde{\varphi}$. By definition, there are markings $f\colon \rH^2(X,\ZZ)\cong \Lambda$ and $g\colon \rH^2(Y,\ZZ)\cong \Lambda$ such that the composition
\[ \Lambda_\QQ \xrightarrow{f^{-1}} \rH^2(X,\QQ)\xrightarrow{\varphi} \rH^2(Y,\QQ)\xrightarrow{g}\Lambda_\QQ \]
is a reflection $\bR_b$, with $b\in \Lambda$ a primitive vector. 

Let $B=\frac{f^{-1}(b)}{n}\in \rH^2(X,\QQ)$ and $B'=\frac{g^{-1}(b)}{n}\in \rH^2(Y,\QQ)$, where $n=\frac{b^2}{2}$.  The map
\[
\widetilde{\varphi} \colon \widetilde{\rH}(X,\mathbb{Z};B) \to \widetilde{\rH}(Y,\mathbb{Z};B'),
\]
defined by sending a vector $(r,c,s)$   to $(n(B,c)-r-ns, \varphi(c)-n((B,c)-s)B',  -s)$ is a Hodge isometry. In particular,
\[
\begin{aligned}
    (0, c, (B,c)) &\mapsto (0, \varphi(c), (B', \varphi(c)),\\
    (0,0,1) &\mapsto (-n, -n B', -1),
\end{aligned}
\]
which gives last assertion.
\end{proof}
 
The following result characterizes the reflective Hodge isometries between abelian surfaces. The idea of the proof is based on \cite[Theorem 1.1]{Hu19}, along with some necessary modifications for abelian surfaces.
\begin{theorem}\label{thm:reflectivetwistedderived}
Let $X$ and $Y$ be two complex abelian surfaces. If there is a reflective Hodge isometry
\[
\varphi \colon \rH^2(X,\mathbb{Q}) \xrightarrow{\sim} \rH^2(Y,\mathbb{Q}),
\]
 then up to sign, $\varphi$ is induced (in the sense of \S\ref{sec:cohomologyrealization}) by a  derived isogeny
\begin{equation}\label{eq:ref-de}
    \D^b(X) \sim \D^b(Y).
\end{equation}

\end{theorem}

\begin{proof}
According to Lemma \ref{lemma:reflectiveHodgeisometry}, there is a Hodge isometry $$\widetilde{\varphi} \colon \widetilde{\rH}(X,\mathbb{Z};B) \xrightarrow{\sim} \widetilde{\rH}(Y,\mathbb{Z};B'),$$
whose restriction on $\rH^2(X,\QQ)$ is just $\varphi$.
Let $v_{B'} = (-n, -n B', -1)$ be the image of the Mukai vector $(0,0,1)$ under $\widetilde{\varphi}$. 
From our construction, the Mukai vector \[v= \exp(-B')\cdot v_{B'} = (-n,0,0) \in \widetilde{\rH}(Y,\mathbb{Z})\] satisfies $v_{B'} = \exp(B') \cdot v$. We can assume that $v$ is positive (see Definition \ref{def:positive}) up to the shift of $\rD^{(1)}(\sY)$. 

Let $\sY \to Y$ be a $\mathbb{G}_m$-gerbe  which admits a $\bB$-field lift $B'$. For some $v$-generic polarization $H$, the moduli stack $\srM_H(\sY,v)$ of $\sY$-twisted sheaves on $Y$ with Mukai vector $v$ forms a $\mathbb{G}_m$-gerbe on its coarse moduli space $M_{H}(\sY,v)$. Let $\cE$ be a universal $(1,1)$-twisted sheaf on $\sY \times \srM_H(\sY,v)$. It induces a twisted Fourier--Mukai transform \[     \Phi^{\cE} \colon \D^{(-1)}(\srM_{H}(\sY,v)) \to \D^{(1)}(\sY), \] (\cf \cite[Theorem 4.3]{yo06}) and a Hodge isometry \[
\varphi^{\cE} \colon \widetilde{\rH}(M_H(\sY,v),\mathbb{Z};B'') \xrightarrow{\sim}\widetilde{\rH}(Y ,\mathbb{Z};B'),
\]
where $B''$ is a $\bB$-field lift of $\srM_{H}(\sY,v)^{(-1)}\to M_{H}(\sY,v)$.
The composition \begin{equation}\label{mod-iso}
(\varphi^{\cE})^{-1} \circ \widetilde{\varphi} \colon \widetilde{\rH}(X,\mathbb{Z};B) \xrightarrow{\sim} \widetilde{\rH}(M_H(\sY,v),\mathbb{Z};B''),\end{equation}
defines a Hodge isometry, which maps the Mukai vector $(0,0,1)$ to $(0,0,1)$ and preserves the Mukai pairing. In addition, it sends $(1,0,0)$ to $(1,b,\frac{b^2}{2})$ for some $b \in \rH^2(Y,\mathbb{Z})$. Changing $B''$ by $B'' + b$, one can obtain a Hodge isometry that simultaneously maps $(1,0,0)$ to $(1,0,0)$ and $(0,0,1)$ to $(0,0,1)$. This restricts to a Hodge isometry
\begin{equation}\label{eq:Hodgeisometry}
\rH^2(X,\mathbb{Z}) \xrightarrow{\sim} \rH^2(M_{H'}(\sY,v),\mathbb{Z}).
\end{equation}

If Hodge isometry \eqref{eq:Hodgeisometry} is admissible, then we can apply Shioda's Torelli Theorem to the abelian surfaces (Theorem \ref{thm:shiodaoriginal}) to conclude that there is an isomorphism \[
f\colon M_{H'}(\sY,v) \xrightarrow{\sim} X\]
such that $(\varphi^{\cE})^{-1} \circ \widetilde{\varphi} = f^*$ up to sign. Take $\sX\to X$  as the $\GG_m$-gerbe  $\sM_{H'}(\sY,v)^{(-1)}\to M_{H'}(\sY,v)$.  Then the Hodge realization of the derived equivalence
\begin{equation}\label{eq:DerivedIso1}
\Phi^{\cE} \circ f^* \colon \rD^{(1)}(\sX) \xrightarrow{\sim} \rD^{(1)}(\sY)
\end{equation}
is $\widetilde{\varphi}$ up to sign. 

Otherwise, the composition
\[
\rH^2(\widehat{X},\ZZ) \xrightarrow{-\D} \rH^2(X, \ZZ) \xrightarrow{\sim} \rH^2(M_H(\sY,v),\ZZ)
\]
is admissible as explained in Example \ref{ex:poincareisomorphism}, which can be realized as the pull-back under an isomorphism $f\colon M_{H}(\sY,v) \xrightarrow{\sim} \widehat{X}$ up to sign. Thus, the Hodge realization of derived equivalence $f^* \circ \Phi^{\cP} \colon \rD^b(X) \xrightarrow{\sim} \rD^b(M_H(\sY,v))$ yields Hodge isometry \eqref{eq:Hodgeisometry}, where $\cP$ is the Poincar\'e bundle. We can consider the following derived isogeny
\begin{equation}\label{eq:DerivedIso2}
    \begin{aligned}
        \D^b(X) \xrightarrow{f^* \circ \Phi^{\cP}} &\D^b(M_H(\sY,v)) & \\
        & \D^{(-1)}(\sM_H(\sY,v)) \xrightarrow{\Phi^\cE} \D^{(1)}(\sY).
    \end{aligned}
\end{equation}
From the construction, its rational Hodge realization on second cohomology yields $\varphi$ up to sign.
\end{proof}

\begin{remark}\label{rmk:primetop}
If $\varphi$ is induced from a reflection of a vector with norm $2n$,  let $\sX\to X$ and $\sY\to Y$ be the equivalent twisted abelian surfaces obtained in Theorem \ref{thm:reflectivetwistedderived}. Then we have \[ [\sX]^n = \exp(n B)  =1  \in \Br(X),\] which implies $[\sX] \in \Br(X)[n]$. Similarly, the order of $[\sY]$ divides $n$. 
\end{remark}

Next,  we are going to show that any rational Hodge isometry can be decomposed into a chain of reflective Hodge isometries. This is a special case of Cartan--Dieudonn\'e theorem which says that any element $g \in \SO(\Lambda_{\mathbb{Q}})$ can be decomposed as products of reflections:
\begin{equation}\label{eq:decomposition}
g = \bR_{b_1} \circ \bR_{b_2} \circ \cdots \circ \bR_{b_n},
\end{equation}
such that $b_i \in \Lambda$, and $(b_i)^2\neq 0$. From the surjectivity of period map \cite[Theorem II]{Shioda78},  for any  rational Hodge isometry
\[
\rH^2(X, \mathbb{Q}) \xrightarrow{\sim} \rH^2(Y,\mathbb{Q}),
\]
we can find a sequence of abelian surfaces $\set{X_{i}}$ with $\Lambda$-markings  and Hodge isometries $$\varphi_{i} \colon \rH^2(X_{i-1},\mathbb{Q}) \xrightarrow{\sim} \rH^2(X_{i},\mathbb{Q}),$$ where $X_{0}=X$ and $X_{n}=Y$, such that $\varphi_{i}$ is induced by some reflection $\bR_{b_i}\in \rO(\Lambda\otimes \QQ)$. We can arrange them as \eqref{eq:zigzagtwsted}:
\begin{equation}\label{eq:zigzagHodge}
\begin{tikzcd}[row sep=0cm, column sep=small]
\rH^2(X,\mathbb{Q})\ar[r,"\varphi_{1}"] &\rH^2(X_1,\mathbb{Q}) &  &\\
& \rH^2(X_1,\mathbb{Q}) \ar[r,"\varphi_{2}"]&  \rH^2(X_2,\mathbb{Q}) &\\
& & \vdots & \\
& & \rH^2(X_{n-1},\mathbb{Q})\ar[r,"\varphi_{{n}}"] & \rH^2(Y,\mathbb{Q}).
\end{tikzcd}
\end{equation}

As a consequence, we get 

\begin{corollary}\label{cor:hodge-derived}
If there is a rational Hodge isometry $\varphi:\rH^2(X,\mathbb{Q}) \xrightarrow{\sim} \rH^2(Y,\mathbb{Q})$, then there is a derived isogeny from $X$ to $Y$, which induces $\varphi$ up to sign as in \eqref{eq:zigzagHodge}. 
\end{corollary}

\begin{remark}

An application of Corollary \ref{cor:hodge-derived} is that any rational Hodge isometry between abelian surfaces is algebraic, which is a special case of Hodge conjecture on product of two abelian surfaces. 
Unlike the case of K3 surfaces, the Hodge conjecture for product of abelian surfaces was known for a long time. See, for example, \cite[Theorem 3.15]{Ra08}. 
\end{remark}

\begin{corollary}\label{cor:HodgeKummer}
There is a rational Hodge isometry $\rH^2(X,\QQ) \xrightarrow{\sim} \rH^2(Y,\QQ)$ if and only if there is a derived isogeny from $\Km(X)$ to $\Km(Y)$.
\end{corollary}
\begin{proof}
Any rational Hodge isometry induces a rational isometry of N\'eron--Severi lattice $\NS(X)_{\QQ} \simeq \NS(Y)_{\QQ}$. Let $\rT(-)$ be the transcendental part of $\rH^2(-)$.
 Applying Witt's cancellation theorem, we can see
\[
\rH^2(X,\QQ) \simeq \rH^2(Y,\QQ) \Leftrightarrow \rT(X)_{\QQ} \simeq
 \rT(Y)_{\QQ},
\]
as Hodge isometries. According to \cite[Theorem 0.1]{Hu19}, $\Km(X)$ is derived isogenous to $\Km(Y)$ if and only if there is a Hodge isometry $\rT(\Km(X))_\QQ\simeq \rT(\Km(Y))_\QQ$. 
Then the statement is clear from the fact that there is a canonical integral Hodge isometry $
\rT(X)(2) \simeq \rT(\Km(X))$  (\cf\cite[Proposition 4.3(i)]{morrison84}).
\end{proof}

\subsection{prime-to-$\ell$ Hodge isometries}
\begin{definition}
We say that a rational Hodge isometry 
\[
\varphi \colon \rH^2(X,\QQ) \xrightarrow{\sim} \rH^2(Y,\QQ)
\]
is prime-to-$\ell$ if it  descends to an isometry $\rH^2(X, \ZZ_{(\ell)}) \xrightarrow{\sim} \rH^2(Y,\ZZ_{(\ell)})$. 
\end{definition}

An easy observation is 

\begin{lemma}\label{lemma:primetoellHodgereflective}
Assume $\varphi:\rH^2(X,\QQ)\xrightarrow{\sim} \rH^2(Y,\QQ)$ is a reflective Hodge isometry, induced by a primitive vector $b \in \Lambda$. Then $\varphi$ is   \emph{prime-to-$\ell$} if and only if $\ell \nmid n= \frac{(b)^2}{2}$.
\end{lemma}
\begin{proof}
One direction is obvious. For the other, suppose $\varphi$ is prime-to-$\ell$.  
 By definition, there are markings  $\rH^2(X,\ZZ)\cong \Lambda$ and  $\rH^2(X,\ZZ)\cong \Lambda$ such that the isometry
\[ \Lambda\otimes \QQ\cong \rH^2(X,\QQ)\xrightarrow{\varphi} \rH^2(Y,\QQ)\cong \Lambda\otimes \QQ\]
is the reflection $\bR_b\in \rO(\Lambda\otimes \QQ)$. As $\varphi$ is prime-to-$\ell$, the reflection $\bR_b$ is lying in $\rO(\Lambda\otimes \ZZ_{(\ell)})$.  

If $\ell \mid n$,  one must have $\ell \mid (x,b)$ for any $x\in \Lambda$. However, this is contradictory, as $\Lambda$ is unimodular and any primitive vector has divisibility $1$. 
\end{proof}

Another useful tool is as follows. 

\begin{lemma}[prime-to-$\ell$ Cartan--Dieudonn\'e decomposition] \label{CDP}Let $\Lambda$ be an integral lattice over $\ZZ$ whose reduction mod $\ell$ is still non-degenerate. 
Any orthogonal matrix $A\in \Og(\Lambda)(\ZZ_{(\ell)}) \subset \Og(\Lambda)(\QQ)$, with $(\ell >2)$, can be decomposed into a sequence of prime-to-$\ell$ reflections. 
\end{lemma}
\begin{proof}
To prove the assertion, we will follow the proof of \cite{Sch50} to refine Cartan--Dieudonn\'e decomposition for any field of characteristic $\neq 2$. In general,  if $\Lambda_k$ is a quadratic space on a field $k$ of characteristic $\neq 2$ with the Gram matrix $G$,  let $I$ be the identity matrix.  

The proof of  Cartan--Dieudonn\'e decomposition  in \cite{Sch50} relies on the following facts: for any element $A\in \Og(\Lambda_k)$, we have
\begin{enumerate}
    \item[i)]  $A$ is a reflection if $\rank (A-I)=1$ (\cf\cite[Lemma 2] {Sch50});
    \item[ii)] Suppose that $\rank (A-I) > 1 $. If $S=G(A-I)$ is not skew symmetric, then there exists $a\in \Lambda$  satisfying $a^t Sa \neq 0$ and $$S+S^t\neq \frac{1}{a^tS a}( Sa\cdot a^t S+S^ta\cdot a^tS^t).$$ In this case $\rank (A \bR_b-I)=\rank (A-I)-1$  and $G(A\bR_b-I)$ is not skew symmetric  with $b=(A-I)a$ satisfying $b^2=-2a^tS a$  (\cf\cite[Lemma 4, Lemma 5]{Sch50}).
    \item [iii)] If $S=G(A-I)$ is skew symmetric, then there exists $b\in \Lambda$ such that $G(A\bR_b-I)$ is not skew symmetric (\cf the proof of \cite[Theorem 2]{Sch50}). 
\end{enumerate}
Then we can decompose $A$ as a series of reflections using ii) repeatedly. In our case, it suffices to show that if $k=\QQ$ and $A$ is coprime to $\ell$, i.e. $nA$ is integral for some $n$ coprime to $\ell$, then
\begin{enumerate}
    \item[i')]  $A$ is a prime-to-$\ell$ reflection if $\rank (A-I)=1$;
    \item[ii')]  Suppose that $\rank (A-I) > 1 $. If the matrix $S=G(A-I)$ modulo $\ell$ is not skew symmetric,  then there exists a vector $a\in \Lambda$ satisfying $\ell \nmid a^t Sa  $, and   
    \begin{equation*}
        S+S^t\neq \frac{1}{a^tS a}( Sa\cdot a^t S+S^ta\cdot a^tS^t).
    \end{equation*}
In this case, $\bR_b$ is prime-to-$\ell$  with $b=(A-I)a$,  $\rank (A \bR_b-I)=\rank (A-I)-1$ and  $G(A \bR_b-I)$ is not skew symmetric;  
    \item [iii')] If the matrix $S=G(A-I)$  modulo $\ell$ is skew symmetric, then there exists $b\in \Lambda$ such that $AR_b$ is coprime to $\ell$ and the modulo $\ell$ reduction of $G(AR_b-I)$ is not skew symmetric. 
\end{enumerate}

For i'), this is obvious.

For ii'), if the modulo $\ell$ reduction $\bar{G}(\bar{A}-\bar{I})$ of $G(A-I)$ is not skew symmetric, we can apply ii) to the  matrix $\bar{A}\in \rO(\Lambda_{\FF_\ell})$ to obtain  a non-zero vector  $\bar{a}\in \Lambda_{\FF_\ell}$ such that $\bar{a}^t\bar{S} \bar{a}\neq 0 \in \FF_\ell$ and 
\begin{equation}\label{eq:red-dec}
        \bar{S}+\bar{S}^t\neq \frac{1}{\bar{a}^t\bar{S} \bar{a}}( \bar{S}\bar{a}\cdot \bar{a}^t \bar{S}+\bar{S}^t\bar{a}\cdot \bar{a}^t\bar{S}^t).
    \end{equation}
Let $a\in \Lambda$ be a lifting of $\bar{a}$. It is easy to see that this is as desired. 

For iii'), the argument is similar to ii'). 
\end{proof} 

As a result, we get the following.
\begin{theorem}\label{thm:prime-hodge-derived}Let $\ell > 2$ be a prime.
If there is a prime-to-$\ell$ rational Hodge isometry $\varphi \colon \rH^2(X,\mathbb{Q}) \xrightarrow{\sim} \rH^2(Y,\mathbb{Q})$, then there exists a prime-to-$\ell$ derived isogeny from $X$ to $Y$, which can induce  $\varphi$ up to sign. Moreover, if $X$ and $Y$ are prime-to-$\ell$ derived isogenus, then there is a prime-to-$\ell$ derived isogeny, in which the orders of $\Gm$-gerbes are all prime-to-$\ell$.
\end{theorem}

\begin{proof}
 By using the  prime-to-$\ell$ Cartan--Dieudonn\'e decomposition given in Lemma \ref{CDP}, one can decompose the Hodge isometry 
 $$\varphi \colon \rH^2(X, \ZZ_{(\ell)}) \xrightarrow{\sim} \rH^2(Y, \ZZ_{(\ell)}),$$ 
 into a chain of prime-to-$\ell$ reflective Hodge isometries. The Lemma \ref{lemma:primetoellHodgereflective} implies that the lift $\widetilde{\varphi}$ extends to an integral isometry
 \[
 \widetilde{\rH}(X, \ZZ_{(\ell)}) \xrightarrow{\sim} \widetilde{\rH}(Y, \ZZ_{(\ell)})
 \]
 In the first case of the proof in Theorem \ref{thm:reflectivetwistedderived}, the derived isogeny \eqref{eq:DerivedIso1} induces $\widetilde{\varphi}$ up to sign, and is thus prime-to-$\ell$. In the second case, the derived isogeny \eqref{eq:DerivedIso2} is also prime-to-$\ell$, since the Poincar\'e dual \[\widetilde{\rH}(X,\ZZ) \xrightarrow{\sim} \widetilde{\rH}(\widehat{X},\ZZ)\] is integral and switches $(0,0,1)$ and $(1,0,0)$.

If $X$ and $Y$ are prime-to-$\ell$ derived isogenous, then there is an isometry $\rT(X) \otimes \ZZ_{(\ell)} \cong \rT(Y) \otimes \ZZ_{(\ell)}$.
Since $\ell >2 $, there is a prime-to-$\ell$ rational Hodge isometry $\rH^2(X, \ZZ_{(\ell)}) \xrightarrow{\sim} \rH^2(Y, \ZZ_{(\ell)})$ by \cite[Theorem 3.2]{Morin-Strom}.
We can use the prime-to-$\ell$ Cartan--Dieudonn\'e decomposition again to obtain a derived isogeny, in which all the reflexive Hodge isometries are  prime-to-$\ell$. Then we can conclude the assertion by Lemma \ref{lemma:primetoellHodgereflective} and Remark \ref{rmk:primetop}.
\end{proof}

\subsection{Isogeny versus derived isogeny}
Let us now describe derived isogenies through suitable isogenies. 

It is well known that the functor $\underline{\Hom}(X,Y)$ of group homomorphisms from $X$ to $Y$ (not just as scheme morphisms) is representable by an \'etale group scheme over $k$ (see \cite[(7.14)]{MoonenGeerBook} for example). Therefore, via Galois descent, we have 
\begin{equation}\label{eq:descentHom}
\Hom_{\abelian_{\bar{k}}}(X_{\bar{k}},Y_{\bar{k}}) \xrightarrow{\sim} \Hom_{\abelian_{\bar{K}}}(X_{\bar{K}},Y_{\bar{K}}),
\end{equation}
for any algebraically closed field $\bar{K} \supset k$.
A similar statement holds for derived isogenies.

\begin{lemma}\label{geo-der-iso}
Let $X$ and $Y$ be abelian surfaces defined over $k$ with $\cha(k)=0$. Let $\bar{K}\supseteq k$ be an algebraically closed field containing $k$. Let $\bar{k}$ be the algebraic closure of $k$ in $\bar{K}$. Then if $X_{\bar{K}}$ and $Y_{\bar{K}}$ are twisted derived equivalent, so are $X_{\bar{k}}$ and $Y_{\bar{k}}$.
\end{lemma}
\begin{proof}
As $X_{\bar{K}}$ is twisted derived equivalent to $Y_{\bar{K}}$, by Theorem \ref{thm1}, there exist finitely many abelian surfaces  $X_0, X_1,\ldots,X_n$ defined over $\bar{K}$  
with $X_0=X_{\bar{K}}$ and 
$$
X_i \cong M_{H_i}(\sX_{i-1}, v_i) \quad Y_{\bar{K}} \cong M_{H_n}(\sX_n, v_n)$$
for some $[\sX_{i-1}]\in\Br(X_{i-1})[r]$.
Let us construct abelian surfaces over $\bar k$ to connect $X_{\bar k}$ and $Y_{\bar k}$  as follows: 

Set $X_0'=X_{\bar{k}}$, then we take $X_1'=M_{H'_1}(\sX_0',v_1')$
where $\sX_0', H'_1$ and $v_1'$ are the descent of $\sX_0, H_1$ and $v$ through the isomorphisms $\Br(X_{\bar{K}})[r]\cong \Br(X_{\bar{k}})[r]$, $\NS(X_{\bar{K}})\cong \NS(X_{\bar{k}})$ and $\widetilde{\rH}(X_{\bar{K}})\cong \widetilde{\rH}(X_{\bar{k}})$.  
The invariance of Brauer group and ($\ell$-adic)Mukai lattice under extension $\bar{k} \subseteq \bar{K}$ is from the smooth base change theorem. For N\'eron--Severi group, see \cite[Proposition 3.1]{MaulikPoonen}.
Then inductively, we can define $X_i'$ as the moduli space of twisted sheaves $M_{H_i'}(\sX_{i-1}', v_i')$ (or its dual, respectively) over $\bar{k}$. Note that we have natural isomorphisms $$(M_{H_i'}(\sX_{i-1}', v_i'))_{\bar{K}}\cong M_{H_i}(\sX_{i-1}, v_i)$$ over $\bar{K}$. In particular,  $(M_{H_i'}(\sX_{n}', v_i'))_{\bar{K}} \cong Y_{\bar{K}}$. It follows that $M_{H_i'}(\sX_{n}', v_i') \cong Y_{\bar k}.$
\end{proof}

For any abelian surface $X_{\CC}$ over $\CC$, the spreading out argument shows that there is a finitely generated field $k \subset \CC$ and an abelian surface $X$ over $k$ such that $X \times_k \CC \cong X_{\CC}$.
We have the following Artin comparison
\begin{equation}\label{eq:artincomparison}
\rH^i_{\et}(X_{\bar{k}},\Zl) \cong \rH^i(X_\mathbb{C},\mathbb{Z}) \otimes_{\mathbb{Z}} \Zl,
\end{equation}
for any $i \in \mathbb{Z}$ and $\ell$ a prime.  Suppose $Y$ is another abelian surface defined over $k$. Suppose $f\colon Y_{\CC}\to  X_{\CC}$ is a prime-to-$\ell$ quasi-isogeny. By definition, it induces an isomorphism of $\ZZ_{(\ell)}$-modules
\[
f^* \colon \rH^1(X_{\CC},\ZZ) \otimes \ZZ_{(\ell)} \xrightarrow{\sim} \rH^1(Y_{\CC},\ZZ) \otimes \ZZ_{(\ell)},
\]
such that there is a commutative diagram
\[
\begin{tikzcd}
\rH^i(X_{\CC},\ZZ) \otimes \ZZ_{(\ell)} \ar[r,"\sim"] \ar[d,hook] & \rH^i (Y_{\CC},\ZZ) \otimes \ZZ_{(\ell)} \ar[d,hook] \\
\rH^i_{\et}(X_{\bar{k}},\Zl) \ar[r,"\sim"] & \rH^i_{\et}(Y_{\bar{k}},\Zl)
\end{tikzcd}
\]
for any $i$, under the comparison \eqref{eq:artincomparison}.
For the converse, we have the following simple fact given by a faithfully flat descent of modules along $\ZZ_{(\ell)} \hookrightarrow \Zl$ and the $\ell$-adic Shioda thick.
\begin{lemma}\label{prime-to-p-iso}
A (quasi-)isogeny $f \colon Y_{\CC} \to X_{\CC}$ is prime-to-$\ell$ if and only if it induces an isomorphism of integral $\ell$-adic realizations
\[
f^* \colon \rH^2_{\et}(X_{\bar{k}}, \Zl) \xrightarrow{\sim}\rH^2_{\et}(Y_{\bar{k}},\Zl).
\]
\end{lemma}

Inspired by Shioda's trick for Hodge isogenies \ref{prop:hodgeisogeny}, we introduce the following notions.
\begin{definition}\label{def:principalquasiiso}
Let $X$ and $Y$ be $g$-dimensional abelian varieties over $k$.
\begin{itemize}
    \item $X$ and $Y$ are (prime-to-$\ell$) \emph{principally isogenous} if there is a (prime-to-$\ell$) isogeny $f$ from $X$ to $Y$ of square degree, that is, $\deg(f)=d^2$ for some $d\in \ZZ$.  This $f$ is called a \emph{principal isogeny}.

    \item  An isogeny $f:X\to Y$ is \textit{quasi-liftable} if $f$ can be written as the composition of finitely many isogenies that are liftable to characteristic zero. 
\end{itemize}

\end{definition}
Now, we can state the main result in this section, which yields in particular Theorem \ref{TDGC}.
\begin{theorem}\label{thm:twistedisogenychar0}
Suppose $\mathrm{char}(k)=0$. Let $\ell >2$ be a prime. The following statements are equivalent:
\begin{enumerate}
    \item $X$ is (prime-to-$\ell$) principally isogenous to $Y$ over $\bar{k}$.
    \item $X$ and $Y$ are (prime-to-$\ell$) derived isogenous over $\bar{k}$.
\end{enumerate}
\end{theorem}
\begin{proof} 
$(1) \Rightarrow (2)$:  we can assume that $f:X\to Y$ is a principal isogeny defined over a finitely generated field $k'$. By embedding $k'$ into $\CC$,  two complex abelian surfaces $X_\CC$ and $Y_\CC$ are derived isogenous since there is a rational Hodge isometry $$\frac{1}{n} f^\ast\otimes \QQ:\rH^2(Y_\CC,\ZZ)\otimes \QQ \cong \rH^2(X_\CC,\ZZ)\otimes \QQ $$  where $\deg(f)=n^2$.  By Lemma \ref{geo-der-iso}, one can conclude that $X_{\bar{k}}$ and $Y_{\bar{k}}$ are derived isogenous, with rational Hodge realization $\frac{1}{n} f^\ast \otimes \QQ$.

If $f$ is a  prime-to-$\ell$ isogeny, the map $\frac{1}{n} f^\ast $   restricts to an isomorphism $$\rH^2(Y_{\CC},\ZZ) \otimes \ZZ_{(\ell)} \xrightarrow{\sim} \rH^2(X_{\CC},\ZZ)\otimes \ZZ_{(\ell)}.$$  The assertion then follows from Theorem \ref{thm:prime-hodge-derived}.

To deduce $(2) \Rightarrow (1)$,  we may assume $X$ and $Y$ are derived isogenous over a finitely generated field $k'$. Embedding $k'$ into $\CC$, $X_\CC$ and $Y_\CC$ are derived isogenous as well by Lemma \ref{geo-der-iso}. According to Remark \ref{rmk:H2realization}, there is a Hodge isometry
\begin{equation}\label{eq:HodgeH2}
\varphi \colon \rH^2(Y_{\CC},\QQ) \xrightarrow{\sim} \rH^2(X_{\CC},\QQ).
\end{equation}
According to Example \ref{ex:poincareisomorphism}, we can assume $\varphi$ is admissible after replacing $X$ by its dual $\widehat{X}$.
By Proposition \ref{prop:hodgeisogeny}, they are principally isogenous over $\CC$. It follows that $X$ and $Y$ are principally isogenous over $\bar{k}$ by \eqref{eq:descentHom}.

If $\D^b(X) \sim \D^b(Y)$ is prime-to-$\ell$, then we can choose a motive isomorphism $\frh^2(X) \simeq \frh^2(Y)$ whose $\ell$-adic realization $\varphi_{\ell}$ is integral by the cancellation theorem over $\ZZ_{\ell}$ (see \cite[Theorem 92:3]{MearaLattice}).  The principal isogeny that induces $\varphi$ is prime-to-$\ell$ by Lemma \ref{prime-to-p-iso}. This proves the assertion. 
\end{proof}

\subsection{Proof of Corollary \ref{cor:TDGC}}
\label{subsec:proofTDGC}
Let us summarize all the results which conclude Corollary \ref{cor:TDGC}. Using an argument similar to the one in Theorem \ref{thm:twistedisogenychar0}, we can reduce them to the case $k=\CC$.
\subsubsection*{$(i)\Leftrightarrow (ii) $}  This is Theorem \ref{thm:twistedisogenychar0}. 

\subsubsection*{$(i)\Leftrightarrow (vi)$} This is Corollary \ref{cor:hodge-derived}.

\subsubsection*{$ (vi)\Leftrightarrow (vii) \Leftrightarrow (viii)$} It follows from the Witt cancellation Theorem.

\subsubsection*{$(i)\Leftrightarrow (iii) $}   This is Corollary \ref{cor:HodgeKummer}.

\subsubsection*{$(ii)\Rightarrow (iv)\Rightarrow (v) $}  This is  from the computation in \cite[Proposition 4.6]{LV19}. In fact, one may take the correspondence
\[\Gamma\coloneqq \bigoplus_i \Gamma_{2i} \colon \frh^{even}(X) \xrightarrow{\sim} \frh^{even}(Y),\]
where
\[\Gamma_{2i} \coloneqq \frac{1}{n^i} f^* \circ \pi_X^{2i} \colon \frh^{2i}(X) \to \frh^{2i}(Y),
\]
and $f \colon X \to Y$ is the given principal isogeny.

\subsubsection*{$(v)\Rightarrow (ii) $} 
Let $\Gamma \colon \frh^{\even}(X) \xrightarrow{\sim} \frh^{\even}(Y)$ be an isomorphism of Frobenius algebra objects.
The Betti realization of its second component is a Hodge isometry by the Frobenius condition (\cf\cite[Theorem 3.3]{LV19}). Thus, $X$ and $Y$ are derived isogenous by Corollary \ref{cor:hodge-derived}, and hence are principally isogenous.

\section{Derived isogeny in positive characteristic}\label{section:quasiisogenycharp}
In this section, we prove the twisted derived Torelli theorem for abelian surfaces over odd characteristic fields. The primary strategy is to lift everything to characteristic zero.  Throughout this section,  we let $k$ denote an algebraically closed field with characteristic $p>3$.

\subsection{Lifting of derived isogenies and quasi-isogenies}\label{subsec:derivedisomixed}  
 Let us start with a lifting result for derived isogenies, which is the only place we may require $p>3$.

\begin{proposition}\label{prop:liftingprimetop}
Let $\sX_0\to X_0$ and $\sY_0\to Y_0$ be twisted abelian surfaces over $k$, which are of finite height. If there is a derived equivalence $\Phi_0\colon \rD^{(1)}(\sX_0) \to \rD^{(1)}(\sY_0)$, then there exists a discrete valuation ring $V$  whose residue field is $k$ and twisted abelian surfaces
\begin{center}
\begin{tikzcd}
\sX_V \ar[r]\ar[rd] & X_V \ar[d]  \\
~ & \Spec (V) 
\end{tikzcd}  ~ and  ~\begin{tikzcd}
\sY_V \ar[r]\ar[rd] & Y_V \ar[d] \\
~ & \Spec (V)
\end{tikzcd} 
\end{center}
over $V$ so that 
\begin{itemize}
    \item the special fibers  are  geometrically isomorphic to $\sX_0\to X_0$  and $\sY_0 \to Y_0 $ respectively. 
    \item there is a Fourier--Mukai transform $\Phi_V \colon \rD^{(1)}(\sX_V)\to \rD^{(1)}(\sY_V)$ whose Fourier-Mukai kernel restricting to $\sX\times \sY $ induces $\Phi_0$. 
\end{itemize}
Moreover, if $\Phi_0$ is prime-to-$p$ and $p>3$, the derived equivalence $\Phi_K\colon \rD^{(1)}(\fX_K)\to \rD^{(1)}(\fY_K)$ on the generic fiber is also prime-to-$p$ where $K$ is the fraction field of $V$.  
\end{proposition}

\begin{proof}
The proof proceeds similarly to \cite[Theorem 5.8]{BraggYang2021}, which proves the existence of liftings of derived isogenies between K3 surfaces.  By Theorem \ref{thm1}, we know that that $$\sX_0^{(-1)}\cong \srM_H(\sY_0, v)$$ is a moduli stack of $\sY_0$-twisted coherent sheaves for some vector $v\in \widetilde{\rN}(\sY_0)$. 
By Lemma \ref{lemma:liftingtwistab}, we can find a DVR $V$ and a projective lift $\sY_V\to Y_V$ over $V$ such that $\NS(Y_V)\cong \NS(Y_0)$. Let $H_V$ be the element in $\NS(Y_V)$ that extends $H$. Following the description of twisted extended N\'eron--Severi lattice as in Proposition \ref{prop:extendedNeronSeveri}, we can see $\widetilde{\rN}(\sY_V) \cong \widetilde{\rN}(\sY_0)$ and hence the twisted Mukai vector $v$ can be extended over $V$, still denoted by $v$.  

Let $\sX_V^{(-1)}=\sM_{H_V}(\sY_V, v)$ be the relative moduli stack of $\sX_V$-twisted coherent sheaves. The universal object in $\rD^{(-1,1)}(\sX_V \times \sY_V)$ induces a derived equivalence $\Phi_V \colon \rD^{(1)}(\sX_V)\to \rD^{(1)}(\sY_V)$ as desired.

For the last assertion, we need to prove that the $p$-adic realization of $\Phi_K$ is integral. This can be deduced from a similar argument as in the proof of Theorem 1.5 in \cite{BraggYang2021}, based on Cais--Liu's crystalline cohomological description for the integral $p$-adic Hodge theory (\cf\cite{CL19,CL19Err}). Let us sketch the proof. As $\Phi$ is prime-to-$p$, its cohomological realization restricts to an isometry of $F$-crystals
\[
\widetilde{\varphi}_p \colon \rH^{\even}_{\crys}(X_0/W) \simeq \rH^{\even}_{\crys}(Y_0/W)
\]
by our definition. The base extension $\widetilde{\varphi}_p \otimes K$ can be identified with the de Rham cohomological realization of $\Phi_K$ 
\[
\widetilde{\varphi}_K \colon \rH^{\even}_{\dR}(X_K/K) \simeq \rH^{\even}_{\dR}(Y_K/K)
\]
by Berthelot--Ogus comparison (\cf\cite[Corollary 2.5]{BO83} or \cite[Theorem B.3.1]{GilletMessing87}). It also preserves Hodge filtrations. 
Let $S$ be the $p$-completion of the divided power envelope of the pair $(W\llbracket u \rrbracket, \ker(W\llbracket u \rrbracket \to \cO_K))$.
Then the map
\begin{equation}\label{eq:iso-S}
    \widetilde{\varphi}_p \otimes_W S \colon \rH^{\even}_{\crys}(X_0/S) \xrightarrow{\sim} \rH^{\even}_{\crys}(Y_0/S)
\end{equation}
is an isomorphism of strongly divisible $S$-lattices (\cf\cite[\S 4]{CL19}). If $p >3$,  according to \cite[Theorem 5.4]{CL19}, one can apply Breuil's functor on \eqref{eq:iso-S} to see that $\phi_K$ restricts to an $\Zp$-integral $\Gal(\bar{K}/K)$-equivariant isometry $\rH^{\even}_{\et}(X_{\bar{K}},\Zp) \xrightarrow{\sim} \rH^{\even}_{\et}(Y_{\bar{K}}, \Zp)$.
\end{proof}
\begin{remark}
The technical requirement for $p >3$ is needed in \cite[Theorem 4.3 (3),(4)]{CL19}. When $\cO_K = W(k)$ is unramified, this condition can be released to $p>2$ by using Fontaine's result \cite[Theorem 2 (iii)]{Fontaine83}. In general, when $p=3$, a possible approach is to prove Shioda's trick as in \S\ref{section:shiodatrick} for strongly divisible $S$-lattices (\cf\cite[Definition 2.1.1]{Breuil99}), which can reduce the statement to crystalline Galois representations of Hodge--Tate weight one. 
\end{remark}

Next, one can  lift separable isogenies between abelian surfaces.
\begin{proposition}\label{lifting} 
Let $f \colon X_0 \to Y_0$ be a separable isogeny between two abelian surfaces over $k$.  Let $W=W(k)$ be the ring of Witt vectors. Then there exist liftings $X_W \to \Spec(W)$ and $Y_W \to \Spec(W)$ such that isogeny $f$ can be lifted to an isogeny $f_W \colon X_W \to Y_W$ such that $\deg f = \deg f_W$. In particular, every prime-to-$p$ isogeny can be lifted to a prime-to-$p$ isogeny.

\end{proposition}
\begin{proof}
According to \cite[Proposition 11.1]{Oort87}, there is a projective lifting $X_W \to \Spec(W)$ of $X_0$. Given that $f$ is separable, $\ker f \subset X_0$ constitutes a finite \'etale group scheme over $k$, which is liftable. 
Choosing a lifting $G_W \subset X_W$ of $\ker f $, we obtain an isogeny $$f_W \colon X_W \to Y_W \coloneqq X_W/G_W, $$ which serves as a lifting of $f$.  If $f$ is prime-to-$p$, then we have $\ker f_W \subseteq X_W[n]$ for some $n$ that is coprime to $p$. Consequently,  $f_W$ is also prime-to-$p$.
\end{proof}

\subsection{Specialization of prime-to-$p$ derived isogenies}
Next, we shall show that prime-to-$p$ geometrically derived isogenies are preserved under reduction.  The idea is to show that the specialization of a moduli space of stable twisted sheaves on an abelian surface or K3 surface remains a moduli space. 
\begin{theorem}\label{thm:specialization}
Let $V$ be a DVR with residue field $k$ and $K=\mathrm{Frac}(V)$.  Let $X_V \to \mathrm{Spec}(V)$ and $Y_V\to \mathrm{Spec}(V)$ be projective abelian surfaces or K3 surfaces over $\mathrm{Spec}(V)$ satisfying 
\begin{equation}\label{eq:pic-const}
    \NS(X_{\bar{K}})\cong \NS(X_k)
\end{equation} where $X_k$ is the special fiber of $X_V\to \Spec(V)$. If their generic fibers $X_K$ and $Y_K$ are (geometrically) prime-to-$p$ derived isogenies, so are the special fibers $X_k$ and $Y_k$.
\end{theorem}
\begin{proof} 

With Theorem \ref{thm:prime-hodge-derived}, it is sufficient to consider the case where there is a derived equivalence
\[
\Phi_V \colon \D^{(1)}(\sX_{\bar K}) \xrightarrow{\sim}\D^{(1)}(\sY_{\bar{K}})
\]
for some  prime-to-$p$  $\GG_m$-gerbes $\sX_K\to X_K $ and $\sY_K\to Y_K$.  From Theorem \ref{thm1}, we know that there is an isomorphism
\[
\sY_{\bar{K}} \cong \sM_H(\sX_{\bar{K}}, v_{K})^{(-1)},
\] for some twisted Mukai vector $v_K \in \widetilde{\rN}(\sX_K)$ and $H_K\in \NS(X_{\bar{K}})$ being $v$-generic. Up to taking a finite extension, we may assume that everything can be defined over $K$. 
 
 We claim that there exists a $\GG_m$ gerbe $\sX_V\to X_V$ whose restriction to $\Spec(K)$ is $\sX_K\to X_K$. It suffices to show that the class $[\sX_K]\in \Br(X_K)$ can be extended to an element in $\Br(X_V)$. 
 By the Chinese remainder theorem, we may assume $\ord([\sX_K]) = \ell^n$ for some prime $\ell \neq p$.
 For each prime $\ell\neq p$, from the Kummer sequence, we have the following commutative diagram
 \[
 \begin{tikzcd}[column sep = small]
     0 \ar[r] & \Pic(X_V)/\ell^n \ar[r] \ar[d] & \rH^1_{\et}(X_V, \mu_{\ell^n}) \ar[d] \ar[r] & \Br(X_V)[\ell^n] \ar[d] \ar[r]& 0 \\
     0 \ar[r] & \Pic(X_{K})/\ell^n \ar[r] & \rH^1_{\et}(X_{K}, \mu_{\ell^n}) \ar[r] & \Br(X_{K})[\ell^n] \ar[r] & 0
 \end{tikzcd}
 \]
The second vertical morphism is an isomorphism by smooth and proper base change. Therefore, $\Br(X_V)[\ell^n]\to \Br(X_K)[\ell^n]$ is surjective, which proves the claim.

By our assumption \eqref{eq:pic-const}, we can pick extensions $v_V\in \widetilde{\rN}(\sX_V)$ and $H_V\in \Pic(X_V)$ of $v_K$ and $H_{K}$.  Let $\sM_{H_V} (X_V, v_V)\to \Spec(V)$  be the relative moduli space of $H_V$-stable twisted sheaves. Then we have the following commutative diagram 
\[
\begin{tikzcd}[column sep=small]
M_{H_V}(\sX_V,v_V)\ar[d]& M_{H_K}(\sX_K,v_K)\ar[l] \ar[r,"\cong"] \ar[d] & Y_K \ar[r]\ar[d] & Y_V \ar[d] \\
\Spec(V) &\spec(K)\ar[l] \ar[r] & \spec(K) \ar[r] & \Spec (V)
\end{tikzcd}
\]
According to Matsusaka--Mumford \cite[Theorem 1]{MM64}, the isomorphism between the generic fiber can be extended to $\Spec (V)$. In particular, $Y_k$ is isomorphic to $M_{H_k}(\sX_{k},v_{k})$ where $v_k=v_V|_{\Spec~k}$ and $H_k=H_V|_{\Spec~k}$.  It follows that there is a prime-to-$p$ derived equivalence $\rD^{(1)}(\sX_k)\simeq \rD^{(-1)}(\sM_{H_k}(\sX_k,v_k))$. 
\end{proof}

\begin{remark}\label{rmk:specialization}
Our proof fails when the twisted derived equivalence is not prime-to-$p$. This is because if the associated Brauer class $\alpha$ has order $p^n$, the map  $$\Br(X_V)[p^n]\to \Br(X_K)[p^n]$$ may not be surjective (\cf\cite[6.8.2]{Poo17}).  
\end{remark}

\subsection{Proof of Theorem \ref{derivediso}}
When $X$ or $Y$ is supersingular, the assertion follows from Proposition \ref{prop:SupersingularModuli} (2).  So we can assume that $X$ and $Y$ both have finite height.

\subsubsection*{$(i')\Rightarrow (ii') $}
By Proposition \ref{prop:liftingprimetop}, we can find projective liftings $ X_V\to \Spec (V)$ and $Y_V\to \Spec(V)$  of $X$ and $Y$ over some DVR $V$ such that there is a prime-to-$p$ twisted derived equivalence between generic fibers $X_K$ and $Y_K$.

By Theorem \ref{thm:twistedisogenychar0}, the generic fibers $X_K$ and $Y_K$ are geometrically prime-to-$p$ principally isogenous. Up to a finite extension of $K$, we can find a prime-to-$p$ principal isogeny  $f_K \colon X_K \to Y_K$. The N\'eron extension property of smooth models $X_V, Y_V$ (\cite[\S 7.3, Proposition 6]{NeronModels}) ensures that $f_K$ can be extended to an isogeny $$f_V \colon X_V \to Y_V.$$ The restriction $f_k:X\to Y$ over the special fibers is still a principal isogeny and we can conclude that $f_k$ is prime-to-$p$  by using Tate's spreading theorem for $p$-divisible groups (\cf\cite[Theorem 4]{tate67}).

\subsubsection*{$(ii')\Rightarrow (i') $} Suppose that there is an isogeny $f \colon X \to Y$, which is prime-to-$p$ of degree $d^2$.
By Proposition \ref{lifting}, we can lift it to a prime-to-$p$ isogeny of degree $d^2$ over $W$:
\[
f_W \colon X_W \to Y_W.
\]
Set $K=\mathrm{Frac}(W)$. The induced isogeny $f_K$ between the generic fibers is a prime-to-$p$ principal isogeny, which induces a $G_K$-equivariant isometry 
\[
\frac{f_K^\ast}{d} \colon \rH^2_{\et}(Y_{\bar{K}}, \Zp) \xrightarrow{\sim} \rH^2_{\et}(X_{\bar{K}},\Zp).
\]  By Theorem \ref{thm:twistedisogenychar0}, there exists a prime-to-$p$ derived isogeny $ \D^b(X_{\bar{K}}) \sim \D^b(Y_{\bar{K}})$ whose $p$-adic cohomological realization is  $\frac{f_K^\ast}{d} $.  The assertion follows from Theorem \ref{thm:specialization}.

\subsection{Further remarks}From the proof of Theorem \ref{derivediso} $ (i')\Rightarrow (ii')$,  we can see that the lifting-specialization argument also works for non prime-to-$p$ derived isogenies. Thus we have 
\begin{theorem}\label{genderivediso}
 Suppose $X_0$ and $Y_0$ are abelian surfaces over $k$ with finite height. If $X_0$ and $Y_0$ are derived isogenous, then they are quasi-liftable  principally isogenous.
\end{theorem}

Moreover, we believe  that the converse of Theorem \ref{genderivediso} also holds. 
\begin{conjecture}\label{gen-derivedconj}
Two abelian sufaces $X_0$ and $Y_0$ are derived isogenous over $k$ if and only if they are quasi-liftable  principally isogenous. 
\end{conjecture}

For this conjecture, our approach remains valid provided that there is a specialization theorem for non prime-to-$p$ derived isogenies. According to the proof of Theorem \ref{thm:specialization}, it suffices to establish the existence of specialization of Brauer classes of order $p$. Adhering to the notations in Theorem \ref{thm:specialization}, this needs  the restriction map $$\Br(X_V)\to \Br(X_K) $$ is surjective. See Remark \ref{rmk:specialization} for further details. 

\subsection{Derived isogeny for Kummer surfaces}
We now proceed to explore the interrelations between the derived isogenies of abelian surfaces and their associated Kummer surfaces. Using the lifting argument, the following theorem is an immediate consequence of the result in characteristic $0$.

\begin{theorem}\label{kumm-isogeny}
Assume $p>2$. 
 If $X_0$ and $Y_0$ are prime-to-$p$ derived isogenous abelian surfaces over $k$, then the associated Kummer surfaces $\Km(X_0)$ and $\Km(Y_0)$ are prime-to-$p$ derived isogenous. Moreover, if there is a derived equivalence 
 \begin{equation}
     \rD^b(\Km(X_0),\alpha_0)\simeq \rD^b (\Km(Y_0),\beta_0)
 \end{equation}
 with $ \mathrm{ord}(\alpha_0)$ and $\mathrm{ord}(\beta_0)$ prime-to-$p$, then $X$ and $Y$ are prime-to-$p$ derived isogenous. 
\end{theorem}
\begin{proof} For the first assertion, as before, we can  quasi-lift the prime-to-$p$ derived isogeny between $X$ and $Y$ to characteristic $0$. By Theorem \ref{derivediso} and Lemma \ref{prop:liftingprimetop}, their liftings are geometrically prime-to-$p$ derived isogenous. According to \cite[Corollary 4.3]{Stellari2007}, we get that the associated Kummer surfaces are prime-to-$p$ derived isogenous. It follows from Theorem \ref{thm:specialization} that $\Km(X_0)$ and $\Km(Y_0)$ are prime-to-$p$ derived isogenous.

For the last assertion, if $X_0$ and $Y_0$ are supersingular, then $\alpha_0$ and $\beta_0$ are trivial under our assumptions. In this case, the result follows from \cite[Theorem 1.2]{LZ21}. Suppose $X_0$ or $Y_0$ is of finite height (then both are of finite height).  According to \cite[Theorem 5.8]{BraggYang2021},  we can find a DVR $V$ with residue field $k$ and projective twisted K3 surfaces over $V$ \begin{center}
    $(S_V, \alpha_V)\to\Spec(V) ~\hbox{and}~(S_V', \beta_V)\to \Spec(V)$,
\end{center} satisfying that 
\begin{itemize}
\item the special fibers are  $(\Km(X_0), \alpha_0)$  and $(\Km(Y_0),\beta_0)$ respectively, 
    \item the generic fibers $(S_{K}, \alpha_{K})$ and $(S_{K}', \beta_{K})$ are geometrically derived equivalent. 
    \item  $\NS(S_{\bar{K}})\cong \NS(\Km(X_0))$ and $\NS( S'_{\bar{K}})\cong \NS(\Km(Y_0))$. 
\end{itemize}

Note that $\NS(S_K)$ and $\NS(S'_K)$ contain Kummer lattices. As seen in the proof of Lemma \ref{lemma:liftingtwistab},  this implies that there exist projective liftings of $X_0$ and $Y_0$, denoted by $X_V\to \Spec(V)$ and $Y_V\to \Spec(V)$,  such that \begin{center}
    $S_{\bar{K}}\cong \Km(X_{\bar K})$  and  $S'_{\bar{K}}\cong \Km(Y_{\bar K})$.
\end{center} 
Choose an embedding $K\hookrightarrow \CC$,  set $X_\CC=X_K\otimes_K \CC$ and $Y_\CC=Y_K\otimes_K\CC$. Then we have a prime-to-$p$ Hodge isometry
\begin{equation}\label{eq:km-iso}
    \rH^2(\Km(X_\CC), \ZZ_{(p)})\to  \rH^2(\Km(Y_\CC), \ZZ_{(p)})
\end{equation}
induced from the prime-to-$p$ derived equivalence. Based on the Kummer construction, for any abelian surface $X_\CC$,  as $p>2$, there is a natural Hodge isometry
\begin{equation*}
    \rH^2(\Km(X_\CC), \ZZ_{(p)})\cong \rH^2(X_\CC,\ZZ_{(p)}) \oplus  (\Sigma_{X_\CC}\otimes \ZZ_{(p)}),
\end{equation*}
 where $\Sigma_{X_\CC}\cong \bigoplus\limits_{i=1}^{16} \ZZ e_i $  with $(e_i, e_j)=-2\delta_{ij}$ is the Kummer lattice.   Then one can obtain a  Hodge isometry 
\begin{equation*}
    \rH^2(X_\CC, \ZZ_{(p)})\to  \rH^2(Y_\CC, \ZZ_{(p)})
\end{equation*}
 from \eqref{eq:km-iso} through the Witt cancellation procedure.  By Theorem \ref{thm:twistedisogenychar0}, $X_K$ and $Y_K$ are geometrically prime-to-$p$ derived isogenous. The assertion follows from Theorem \ref{thm:specialization}. 
\end{proof}
 
\begin{remark} It is natural to consider if one can apply the lifting method to prove the converse of Theorem \ref{kumm-isogeny}. Specifically,  one may wonder if 
 $\Km(X_0)$ and $\Km(Y_0)$  are prime-to-$p$ derived isogenous, as are $X$ and $Y$. 
 
 However, the issue is that the derived isogeny between $\Km(X_0)$ and $\Km(Y_0)$  is merely quasi-liftable, not known to be liftable. In other words, although we can lift every derived equivalence between twisted abelian surfaces or K3 surface to characteristic $0$,  we cannot necessarily find some liftings of $X_0$ and $Y_0$ respectively such that the generic fibers of their associated Kummer surfaces are prime-to-$p$ geometrically derived isogenous. 
\end{remark}

\bibliographystyle{amsplain}
\bibliography{derivedab_final}
\end{document}